\documentclass[article,10pt]{article}
\usepackage{amsmath}
\usepackage{amsfonts}
\usepackage{array}
\usepackage{enumerate}
\usepackage{graphicx}
\usepackage{amssymb}
\usepackage{amsthm}
\usepackage{xcolor}
\usepackage{chemarr}
\usepackage{mathtools}
\usepackage{tcolorbox}
\usepackage{caption}
\usepackage{booktabs}
\usepackage{makecell}
\usepackage{siunitx}
\usepackage[margin=1in]{geometry}
\usepackage[version=3]{mhchem}
\usepackage{chemfig} 

\newtheorem{proposition}{Proposition}
\newtheorem{lemma}{Lemma}
\newtheorem{corollary}{Corollary}

\newtheorem{remark}{Remark}
\newtheorem{example}{Example}

\newcommand{\norm}[2][\relax]{\ifx#1\relax \ensuremath{\left\Vert#2\right\Vert} \else \ensuremath{\left\Vert#2\right\Vert_{#1}}\fi}

\begin{document}

\title{Rigorous estimates for the quasi-steady state approximation \\ of the Michaelis--Menten reaction
mechanism\\ at low enzyme concentrations}

\author{Justin Eilertsen\\
            Mathematical Reviews\\ American Mathematical Society\\
            416 4th Street\\Ann Arbor, MI, 48103\\
            e-mail: {\tt jse@ams.org}\\\\
        Santiago Schnell\\
            Department of Biological Sciences and\\
            Department of Applied and Computational Mathematics and Statistics\\
            University of Notre Dame\\
            Notre Dame, IN 46556\\
            e-mail: {\tt santiago.schnell@nd.edu}\\\\
        Sebastian Walcher\\
            Mathematik A, RWTH Aachen\\
            D-52056 Aachen, Germany\\
            e-mail: {\tt walcher@matha.rwth-aachen.de}}


\maketitle

\begin{abstract}
 There is a vast amount of literature concerning the appropriateness of various perturbation parameters
 for the standard quasi-steady state approximation in the Michaelis--Menten reaction mechanism, and 
 also concerning the relevance of these parameters for the accuracy of the  approximation by the 
 familiar Michaelis--Menten equation. Typically, the arguments in the literature are based on (heuristic) 
 timescale estimates, from which one cannot obtain reliable quantitative estimates for the error of 
 the quasi-steady  state approximation. We take a different approach. By combining phase plane analysis 
 with differential inequalities, we derive sharp explicit upper and lower estimates  for the duration 
 of the initial transient and substrate depletion during this transitory phase. In addition, we obtain 
 rigorous bounds on the accuracy of the standard quasi-steady state approximation in the slow dynamics 
 regime.  Notably, under the assumption that the quasi-steady state approximation is valid over the 
 entire time course of the  reaction, our error estimate is of order one in the Segel--Slemrod parameter.
\end{abstract}

\section{Introduction}\label{introsec}
We consider the classical Michaelis--Menten reaction mechanism for enzyme action. Its time evolution 
is governed by the ordinary differential equations for the substrate $s$ and intermediate complex $c$ 
concentrations
\begin{equation}\label{eqmmirrev}
\begin{array}{rclclcl}
\dot s&=& -k_1e_0s&+&(k_1s+k_{-1})c,   \\
\dot c&=& k_1e_0s&-&(k_1s+k_{-1}+k_2)c \\
\end{array}
\end{equation}
with initial values $s(0)=s_0$, $c(0)=0$, and conservation laws for the substrate and enzyme \cite{SchMai}. 
We will focus on the standard quasi-steady state (QSS) \cite{SSl} approximation with low initial enzyme 
concentration $e_0$.  In this case, the appropriate reduction is given by the Michaelis--Menten 
equation\footnote{The common choice for the initial value of the Michaelis--Menten equation~\eqref{classmmeq} 
is $s(0)=s_0$. This choice is convenient from the experimental point of view, and also compatible with 
singular perturbation theory, but it needs to be considered critically with regard to parameter 
identification experiments. 
We will discuss this point in the course of the paper.}
\begin{equation}\label{classmmeq}
    \dot s =-\dfrac{k_1k_2e_0s}{k_{-1}+k_2+k_1s}=-\dfrac{v_\infty s}{K_M+s},
\end{equation}
with the Michaelis constant 
\begin{equation}
    K_M:=\dfrac{k_{-1}+k_2}{k_1}
\end{equation}
and the limiting rate
\begin{equation}
    v_\infty=k_2e_0.
\end{equation}
For a biochemical definition of the above constants, we invite the readers to consult \cite{cornish}.

If we look geometrically at the ordinary differential equation system~\eqref{eqmmirrev}, the slow manifold 
(or QSS manifold) is defined by 
\begin{equation}\label{classsmeq}
c=g(s):=\dfrac{k_1e_0s}{k_{-1}+k_2+k_1s}=\dfrac{e_0s}{K_M+s}.
\end{equation}
The accuracy and the range of validity for the QSS reduction are not only of theoretical 
interest, but also of practical relevance for parameter identification in the laboratory. Ideally, 
practitioners wish for a suitable ``small parameter'' that ensures accuracy of the reduction, while 
measurements are taken in laboratory experiments.\footnote{The Michaelis--Menten equation is also
used for modeling biochemical reactions in signaling, metabolic and pharmacological pathways.}  In 
his seminal paper, Segel~\cite{Seg} \textcolor{black}{derived a parameter by comparing two (in part 
heuristically determined) timescales; this parameter} is widely accepted. However, the 
arguments used to derive this parameter, like several variants \cite{SSl,borghans,SchMai,tza}, 
cannot provide a quantitative estimate for the approximation error. Indeed, there seem to be no 
rigorous and meaningful quantitative estimates available in the literature (see, \cite{natural} 
for a more detailed account). Moreover, one can estimate $K_M$ and $v_\infty$, by fitting 
experimental data to the Michaelis--Menten equation~\eqref{classmmeq} under steady-state assay conditions, 
but obtaining $k_1$ would also be of interest. From a practical perspective, further information 
is needed about the onset time of the QSS regime, and the substrate depletion in the transitory 
phase. \textcolor{black}{Despite its ubiquity in analyzing enzymatic reactions, the Michaelis--Menten 
equation lacks a rigorous mathematical framework for accurately estimating key parameters $K_M$ 
and $v_\infty$ from common laboratory measurements, such as initial rates and progress curves.}

\subsection{Goal and results of the present paper}
The fundamental goal of the present paper is to provide: (i) reasonably sharp and rigorous estimates 
for the approximation error, (ii) the determination of lower and upper estimates for the onset time 
of QSS, and (iii) the substrate loss in the initial transient of the reaction. Our approach is 
inspired by arguments from singular perturbation theory. However, our methods mostly rely on 
elementary facts concerning differential equations and differential inequalities.

In Section~\ref{sec:review}, we recall some qualitative features of \eqref{eqmmirrev}. In particular, 
we recollect that the time $t_{\rm cross}$ when the solution crosses the QSS manifold suffices as 
an onset time for the slow regime. Section~\ref{estsec} contains the rigorous estimates that comprise 
the fundamental technical results of this paper. By modifying a Lyapunov function approach, we first 
obtain upper and lower limits for $t_{\rm cross}$, which is of interest in its own right. Using 
differential inequalities, we then obtain upper and lower limits for the substrate depletion in the 
transitory phase. In a final step, we derive (in two different ways) rigorous bounds for the approximation
error during the QSS regime. Generally, these turn out to be of order $\varepsilon\log(1/\varepsilon)$, where 
$\varepsilon$ here denotes the parameter proposed by Segel~\cite{Seg}. For the special situation 
corresponding to an initial value $s_0$ for the Michaelis-Menten equation at $t=0$, we obtain sharper bounds of 
order $\varepsilon$ over the whole time range.  By nature this is a rather technical section, but the 
technical expenditure also yields estimates for the reliability of (simpler) asymptotic error bounds.
\textcolor{black}{
In Section~\ref{sec:applications}, we list and discuss these asymptotic bounds with a view on their
relevance in laboratory practice. Application-oriented readers may just skim Section~\ref{estsec}, 
and proceed directly to Section~\ref{sec:applications}, which is accessible independently. In Appendix, 
we present a quick overview of parameters relevant for the dynamics and the approximation, and we list 
the relevant results and parameters for the case of small $k_1$.}

\section{Review of qualitative properties} \label{sec:review}
We first recall some qualitative features and some underlying theory. In later sections, we will focus 
less on what these results say, but rather go beyond them towards quantitative results. 
 
\subsection{The standard quasi-steady-state reduction}
The standard quasi-steady-state (sQSS) approximation, as given by \eqref{classmmeq}, is a well-known 
approximation to \eqref{eqmmirrev}. It was originally obtained by Briggs \& Haldane~\cite{BrHa}, 
and put on solid mathematical ground by Heineken, Tsuchiya \& Aris~\cite{hta}, who applied the 
singular perturbation theory developed by Tikhonov~\cite{tikh} and later by Fenichel~\cite{fenichel}. 
By singular perturbation theory, the reduction (\ref{classmmeq}) accounts with high accuracy 
for the depletion of substrate after a short transitory phase, whenever the initial enzyme 
concentration, $e_0$, is sufficiently small with respect to the initial substrate concentration, 
$s_0$. The utility of \eqref{classmmeq} emanates from the fact that initial enzyme and substrate 
concentration are controllable within an experiment. Therefore, it is as least theoretically 
possible to prepare an experiment in a way that ensures \eqref{classmmeq} is an appropriate model 
from which to estimate the kinetic parameters: $K_M$ and $v_\infty$. However, the phrasing
``sufficiently small $e_0$'' is qualitative, and certainly not sufficient to satisfy a quantitative
experimentalist or even a theorist (in certain contexts). Thus, in any practical application of 
\eqref{classmmeq}, one is forced to ask: How small should $e_0$ be to confidently replace 
\eqref{eqmmirrev} with \eqref{classmmeq}?

Several dimensionless parameters, $\varepsilon_{X}$, have been introduced in the literature that 
suggest (at least implicitly) that the error between \eqref{classmmeq} and \eqref{eqmmirrev} is 
bounded by $\gamma \cdot \varepsilon_{X}$, where $\gamma$ is a dimensional constant with units 
of concentration. From Briggs \& Haldane~\cite{BrHa}, we have
 \begin{equation}\label{bhparam}
\varepsilon_{BH}=\dfrac{e_0}{s_0},
 \end{equation}
which was also employed by Heineken, Tsuchiya \& Aris~\cite{hta}. Other notable dimensionless 
parameters include 
 \begin{equation}\label{rsparam}
\varepsilon_{RS}=\dfrac{k_1e_0}{k_{-1}+k_2}=\dfrac{e_0}{K_M},
 \end{equation}
 originally proposed by Reich and Selkov \cite{ReSe}, as well as the widely used
\begin{equation}\label{sslparam}
\varepsilon_{SSl}=\dfrac{k_1e_0}{k_{-1}+k_2+k_1s_0}=\dfrac{e_0}{K_M+s_0},
\end{equation}
which was introduced by Segel~\cite{Seg} and analyzed by Segel \& Slemrod~\cite{SSl}. Finally, 
we mention
\begin{equation}\label{cspparam}
\varepsilon_{MM}=\dfrac{k_1k_2e_0}{(k_{-1}+k_2)^2}=
    \dfrac{e_0}{K_M}\cdot\dfrac{k_2}{k_{-1}+k_2}=\varepsilon_{RS}\cdot\dfrac{k_2}{k_{-1}+k_2}
\end{equation}
which reflects the linear timescale ratio at the stationary point as $e_0\to 0$, as follows from 
\cite{Chihuahua} Proposition 1 and Remark 2. In particular, see Eq.~(9) in \cite{Chihuahua}.

All the parameters $\varepsilon_X$ mentioned above have the following property. If $\varepsilon_{X}$ 
approaches zero in a well defined manner with $e_0\to 0$, while the other reaction parameters are 
bounded above and below by positive constants, then the $s$ component of the exact solution will 
approach the approximate solution with any degree of accuracy. \textcolor{black}{Moreover, given 
these well-defined conditions\footnote{\textcolor{black}{Such restrictions are needed. For instance, 
one gets $\varepsilon_{SSl}\to 0$ when $s_0\to\infty$ but the approximation by the Michaelis-Menten 
equation \eqref{classmmeq} is incorrect; see \cite{Chihuahua}, Section 5.}}, asymptotically all the 
parameters noted are of the same order, 
e.g. $\varepsilon_{RS}= \dfrac{K_M+s_0}{K_M}\varepsilon_{SSl}=o(\varepsilon_{SSl})$ with the factor 
bounded above and below by positive constants.} 

However, contrary to an assumption 
prevalent in the literature, from these (and other proposed) parameters one cannot obtain 
quantitative information \cite{Chihuahua,natural}. Moreover, expressions like $\varepsilon_{X}\ll1$ 
are sometimes used in a literal interpretation (such as ``$10^{-2}\ll 1$'') 
\cite[for an example]{Schfar}, which misses the point.

Ideally, a dimensionless small parameter $\varepsilon_{\rm ideal}$ should control the discrepancy 
between the $s$ component of the solution of the system~\eqref{eqmmirrev} with initial value $(s_0,0)$ 
and the solution of the approximate equation~\eqref{classmmeq} (with initial value $s_0$), by an 
estimate  $\varepsilon_{\rm ideal}\cdot s_0$. Obtaining such a parameter is a principal goal of 
the present paper.

\subsection{Demarcating fast and slow dynamics of the reaction}
For the initial value $(s,c)(t=0)=(s_0,0)$ for \eqref{eqmmirrev},\footnote{Generally, for any 
initial value below the graph of the slow manifold.} we need to determine a point in time to separate 
fast and slow dynamics. Singular perturbation theory does not provide a unique choice for such a 
point in time. But, as noted by Schauer \& Heinrich~\cite{hs-inv}, and proven in Noethen \& 
Walcher~\cite{NoWa} and Calder \& Siegel~\cite{CaSi},\footnote{Calder and Siegel~\cite{CaSi} also 
proved the existence of a unique distinguished invariant manifold. An extension to the open 
Michaelis--Menten reaction mechanism was given in \cite{ERSW}.} there exists a distinguished time 
for the governing equations~\eqref{eqmmirrev} of the Michaelis--Menten reaction mechanism. We recall 
this fact:

\begin{lemma}\label{schahelem}
The solution of \eqref{eqmmirrev} with initial value $(s_0,0)$ crosses the graph of $g$ at a 
unique positive time $t_{\rm cross}$, and remains above the graph for all $t>t_{\rm cross}$. 
One has $\dot c(t)\geq 0$ for all $t\leq t_{\rm cross}$ and $\dot c(t)\leq 0$ for all 
$t\geq t_{\rm cross}$. Moreover, $\dot s(t)<0$ for all $t\geq 0$.
\end{lemma}

Thus, we note a biochemical property of the reaction. The maximal concentration of complex $c$ 
is attained at $t= t_{\rm cross}$. In view of this, it seems natural to consider 
$t_{\rm cross}$ as a starting time of the slow phase.\footnote{In view of non-uniqueness, this 
designation is not meant to imply that the slow dynamics sets in precisely at $t_{\rm cross}$. We 
invite the readers to see also the discussion in Remarks \ref{lsquarerem} and \ref{lcuberem}.} We 
furthermore set 
$s_{\rm cross}:=s(t_{\rm cross})$ and $c_{\rm cross}:=c(t_{\rm cross})$.

We illustrate the $(s,c)$ phase plane geometry of the Michaelis--Menten reaction mechanism 
in {\sc Figure}~\ref{FIG00}. {\bf Lemma~\ref{schahelem}} shows that the set above the graph of 
$g$ is positively invariant. This result can be sharpened. For $0\leq \delta\leq 1$, set
\begin{equation}\label{gdeltadef}
    g_\delta(s)=\dfrac{k_1e_0s}{(1-\delta)k_2+k_{-1}+k_1s},
\end{equation}
noting that \textcolor{black}{for each $\delta$, $c=g_\delta(s)$ defines an isocline of 
system~\eqref{eqmmirrev}, along which the vector field has a fixed direction. In particular,
$c=g_0=g$ defines the $c$-isocline (thus $\dot c=0$), and $c=g_1$ defines the $s$-isocline (thus 
$\dot s=0$).} Now, set 
\begin{equation}\label{deltastardef}
    \delta^*:=\dfrac{k_1}{2k_2}(K_M+e_0)\left(1-\sqrt{1-\dfrac{4k_2}{k_1}\dfrac{e_0}{(K_M+e_0)^2}}\right).
\end{equation}
In Noethen \& Walcher~\cite[Props.\ 5 and 6, with proofs stated for the $(s,p)$-plane]{NoWa},
the following was shown:
\begin{lemma}\label{nowaestlem}
    For every $\delta\geq \delta^*$, the subset of $[0,\,s_0]\times[0,\,e_0]$ which is bounded 
    by the graphs of $g_0$ and $g_\delta$ is positively invariant for the Michaelis--Menten 
    reaction mechanism system.
\end{lemma}
\begin{remark}{\em 
    The expression for $\delta^*$ may look prohibitive, but less complicated estimates are readily 
    obtained \textcolor{black}{for small enzyme concentration}. For example, given $x\leq 0.1$, by 
    the mean value theorem and generous estimates, there exists $\xi\leq 0.1$ so that 
    \begin{equation*}
        \sqrt{1-x}-1=-\dfrac{1}{2\sqrt{1-\xi}}\, x \leq -0.9\, x,
    \end{equation*}
    from which one sees that
\begin{equation}\label{deltastarest}
    \delta^*\leq \frac{10}{9}\dfrac{e_0}{K_M+e_0}\leq \frac{10}{9}\varepsilon_{RS}\quad
    {\rm whenever} \quad \varepsilon_{RS}\leq 0.1.
\end{equation}
}
\end{remark}

\begin{figure}[htb!]
\centering
\includegraphics[width=10.0cm]{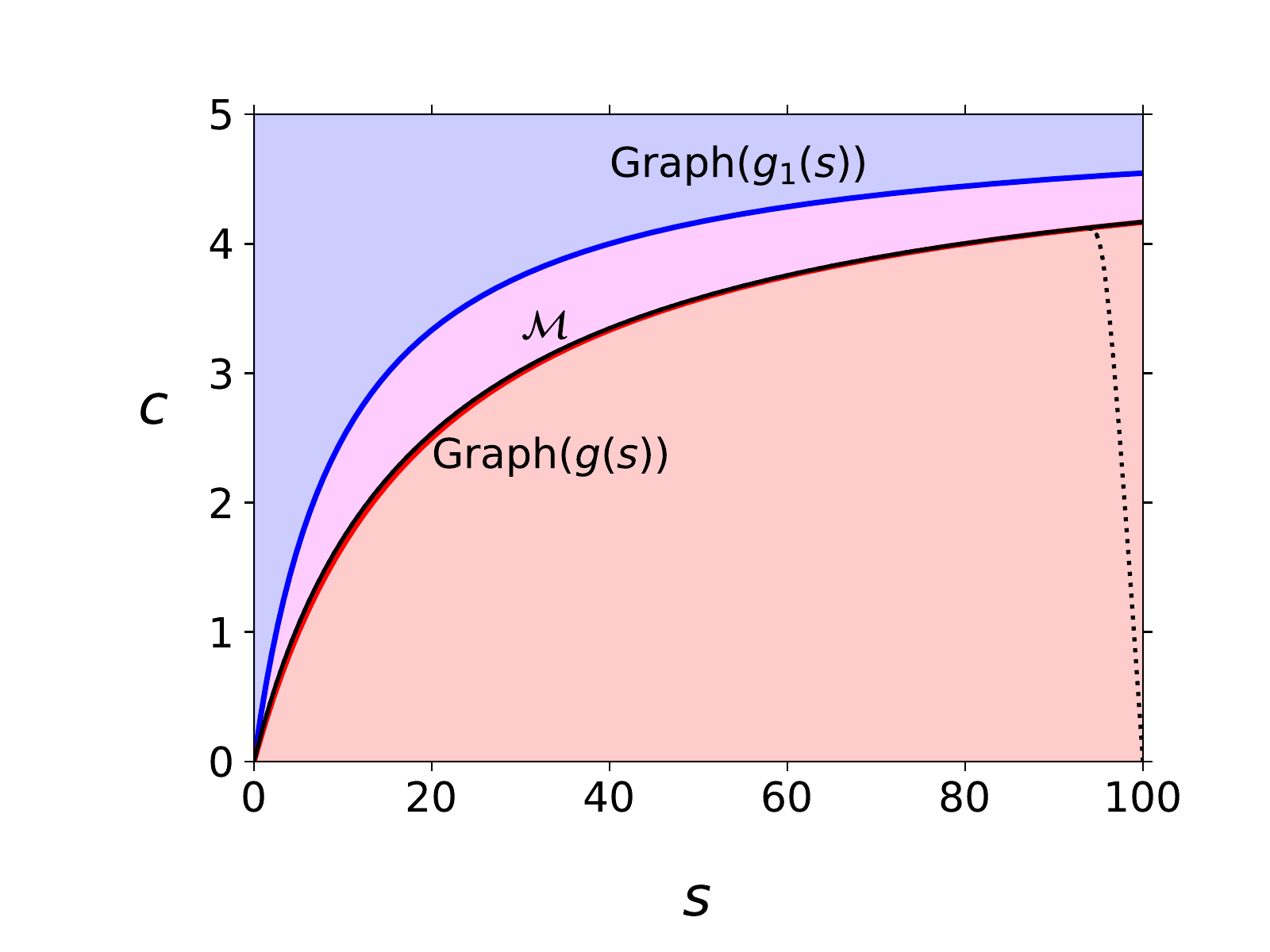}
\includegraphics[width=10.0cm]{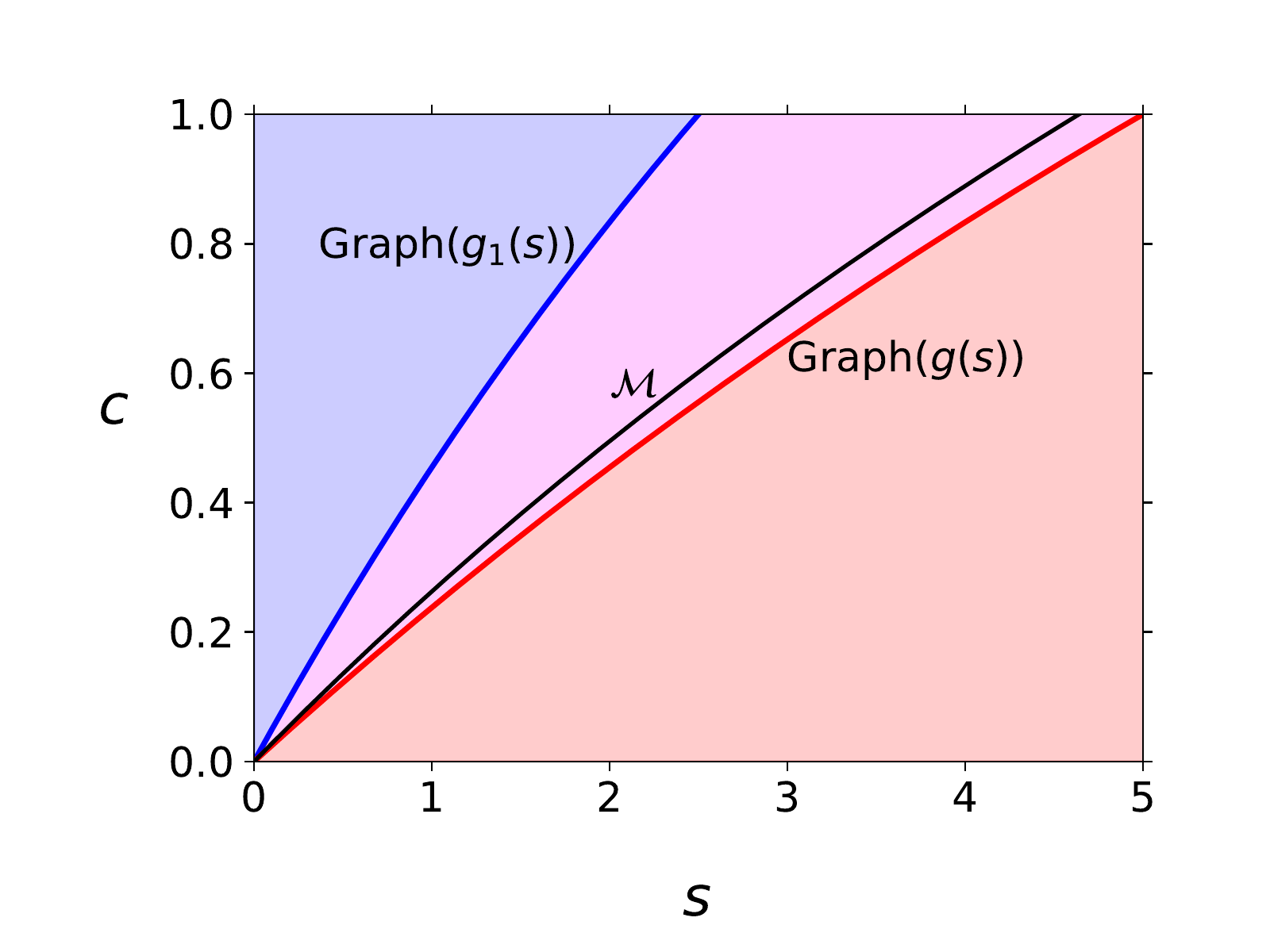}
\caption{\textbf{The $(s,c)$ phase plane geometry of the Michaelis--Menten reaction mechanism.} 
The thick red curve, the graph of $g(s)$, is the QSS variety (i.e., the $c$-nullcline) and the 
thick blue curve, the graph of $g_1(s)$, is the $s$-nullcline. The thick black curve is the invariant 
slow manifold, $\mathcal{M}$. In the shaded violet 
region between the graphs is the slow invariant manifold, $\mathcal{M}$, that connects 
the stable equilibrium at the origin with a saddle equilibrium at infinity. The vector field in 
the red shaded region below Graph($g(s)$) satisfies $\dot{c}>0$ and $\dot{s} <0$. On Graph($g(s)$), $\dot{c}=0$ 
and $\dot{s}<0.$  In the magenta region that lies above Graph($g(s)$) and below Graph($g_1(s)$), $\dot{s}<0$ 
and $\dot{c}<0.$ On $g_1(s)$, $\dot{s}=0$ and $\dot{c}<0$. In the blue shaded region above 
Graph($g_1(s)$) and below $c=e_0$, $0<\dot{s}$ and $\dot{c}<0.$ The dotted black curve in the {{\sc top}} 
panel is a single trajectory obtained via numerical integration of the mass action 
equations~\eqref{eqmmirrev} with parameters (in arbitrary units): $s_0=100,e_0=5.0, k_1=1.0,k_2=k_{-1}=10.0$. 
The trajectory approaches and intercepts the graph of $g(s)$ at $t=t_{\rm cross}.$ For $t>t_{\rm cross}$, 
the trajectory lies above Graph($g(s))$, but below $\mathcal{M}$. {{\sc Top}}: The trajectory 
enters the the magenta region and approaches $\mathcal{M}$ as time evolves forward. 
{{\sc Bottom}}: Closeup of the {{\sc top}} panel near the QSS variety. The trajectory still 
lies below $\mathcal{M}$, but becomes effectively indistinguishable from $\mathcal{M}$ as 
$t\to\infty$.} \label{FIG00}
\end{figure}
\clearpage
\newpage
\section{Critical estimates for the dynamics of the Michaelis--Menten reaction mechanism}\label{estsec}
One can rewrite system \eqref{eqmmirrev} as
\begin{equation}\label{eqmmirrev2}
\begin{array}{rcl}
\dot s&=& -k_1e_0s+(k_{-1}+k_1s)c,   \\
\dot c&=&  -(k_{-1}+k_2+k_1s)(c-g(s)).\\
\end{array}
\end{equation}
In the above system, $g(s)$ is given by \eqref{classsmeq}.
We are only interested in the solution with initial value $(s_0,0)$, which starts below the graph 
of $g$. Since $\dot c\leq 0$ for $c\geq g(s)$, we have
\begin{equation}\label{ctildemmlow}
c\leq \widetilde c:= \max_{0\leq s\leq s_0}g(s)=\dfrac{e_0s_0}{K_M+s_0}=\varepsilon_{SSl}\, s_0.
\end{equation}
We will frequently use basic properties of differential inequalities (see, for instance, 
Walter~\cite[\S9, Theorem 8]{walter}). For later use, we note two estimates for substrate 
concentration, $s$:
\begin{lemma}\label{simplesestlem}
Let $s(t)$ be the first component of the solution of \eqref{eqmmirrev2} with initial value 
$(s_0,0)$. Then, for 
all $t\geq 0$ one has 
\begin{equation}\label{simplesest}
    s(t)\geq s_0 \exp(-k_1e_0 t),
\end{equation}
and
\begin{equation}\label{notsosimplesest}
    s(t) \leq s_0\left( \cfrac{k_{-1}}{k_{-1}+k_2} +\cfrac{k_2}{k_{-1}+k_2} 
                                \exp\bigg(-\cfrac{k_1e_0K_M}{K_M+s_0}\cdot t\bigg)\right).
\end{equation}
\end{lemma}
\begin{proof}
The first estimate follows readily from $\dot s\geq -k_1e_0s$, $s(0)=s_0$. As for the second, 
from the first equation in \eqref{eqmmirrev}, with \eqref{ctildemmlow}, one finds
\[
\dot s=-k_1(e_0-c)s+k_{-1} c\leq -k_1(e_0-\widetilde c)s+k_{-1} \widetilde c.
\]
Comparing $s$ with the solution of the linear differential equation 
\[
\dot x=k_1(e_0-\widetilde c)x+k_{-1} \widetilde c= -\dfrac{k_1e_0K_M}{K_M+s_0}\, x
                                            +\dfrac{k_{-1}e_0s_0}{K_M+s_0},\quad x(0)=s_0
\]
shows the assertion.
\end{proof}

\begin{remark}\label{badsestrem}{\em 
Note that the derivative of the right-hand side of \eqref{simplesest} is equal to $-k_1e_0s_0$ 
at $t=0$, which agrees with $\dot s(0)$ in \eqref{eqmmirrev}, while the right-hand side 
of \eqref{notsosimplesest} has derivative $-\dfrac{k_2}{k_1(K_M+s_0)}\cdot k_1e_0s_0$, which is 
markedly different. From this perspective, the upper estimate is not optimal.}
\end{remark}

We will proceed in three steps. First, we estimate the distance of the solution to the slow 
manifold. In a second step, we obtain lower and upper approximations for $t_{\rm cross}$, and
we compare exact and approximate solutions near the slow manifold in the third step. 
\textcolor{black}{Throughout we will not impose any a priori assumptions concerning smallness of 
parameters --- so most estimates are universally valid if not necessarily sharp  --- but indicate 
when such assumptions are made. }

\subsection{First Step: Approach to the slow manifold}
We will employ two variants of a Lyapunov function approach. The first variant is based on 
established procedure \cite[see, as an example, Section 2.1]{BeGe}. However, some adjustments are 
necessary, because system \eqref{eqmmirrev} with small parameter $e_0=\varepsilon e_0^*$ 
($e_0^*$ some reference value) is not in Tikhonov standard form with separated slow and fast 
variables. We will restrict attention to the compact positively invariant rectangle defined by 
$0\leq s\leq s_0$ and $0\leq c\leq e_0^*$.  By \eqref{ctildemmlow}, we may further restrict 
attention to the rectangle defined by $0\leq s\leq s_0$ and $0\leq c\leq \widetilde c$.
Consider
\begin{equation}\label{quadlyap}
    \cfrac{d}{dt} (c-g(s))^2 = -2(k_{-1}+k_2+k_1s)(c-g(s))^2 - 2g'(s)\dot{s}(c-g(s)).
\end{equation}
Let $L:=c-g(s)$. By invoking $\dot{s}+\dot{c}=-k_2c$, we obtain with repeated use of 
\eqref{eqmmirrev2}: \textcolor{black}{
\begin{subequations}
\begin{align}
    \cfrac{d}{dt} L^2 &= -2(k_{-1}+k_2+k_1s))L^2 -2g'(s)\dot{s}\, L\\
    &= -2(k_{-1}+k_2+k_1s)L^2 +2g'(s)\bigg[\dot{c}+k_2c\bigg]\, L\\
    &= -2(k_{-1}+k_2+k_1s)L^2 + 2g'(s)\bigg[-(k_{-1}+k_2+k_1s)\,L + k_2c\bigg]\,L\\
    &= -2(k_{-1}+k_2+k_1s)L^2 -2(k_{-1}+k_2+k_1s)g'(s)L^2 + 2g'(s) k_2c \,L\label{basicderiv}\\
    &\leq -2\bigg[\min_{s\in [0,s_0]}(k_{-1}+k_2+k_1s)(1+g'(s))\bigg]L^2 
        + 2\max_{s\in [0,s_0]} |g'(s)| \cdot k_2\widetilde c\cdot |L|.
\end{align}
\end{subequations}
}
Now 
\begin{equation}\label{gprime}
g'(s)=\cfrac{K_Me_0}{(K_M+s)^2}\geq 0;\quad \max_{s\in [0,s_0]} |g'(s)|=\cfrac{e_0}{K_M},
\end{equation}
therefore
\[
\min_{s\in [0,s_0]}(k_{-1}+k_2+k_1s)(1+g'(s))\geq \min_{s\in [0,s_0]}(k_{-1}+k_2+k_1s)=k_{-1}+k_2.
\]
Altogether we obtain with \eqref{ctildemmlow} 
\begin{equation}\label{diffineq}
 \cfrac{d}{dt} L^2\leq -2(k_{-1}+k_2)\cdot L^2
        + 2\cfrac{e_0}{K_M}\cdot\dfrac{k_2e_0s_0}{K_M+s_0}\cdot |L|.
\end{equation}
Next, we apply the Cauchy-Schwarz inequality
\[
2\cfrac{e_0}{K_M}\cdot\dfrac{k_2e_0s_0}{K_M+s_0}\cdot |L|\leq \sigma L^2
    +\left(2\cfrac{e_0}{K_M}\cdot\dfrac{k_2e_0s_0}{K_M+s_0}\right)^2\cdot\dfrac{1}{2\sigma},
\]
which holds for any $\sigma>0$.
For $\sigma=k_{-1}+k_2$, this yields
\begin{equation}\label{diffineqvar}
 \cfrac{d}{dt} L^2\leq -(k_{-1}+k_2)\cdot L^2+\frac12 \cdot\ (k_{-1}+k_2)\cdot 
        \left(\varepsilon_{SSl}\cdot\varepsilon_{MM}\right)^2\cdot s_0^2.
\end{equation}

\begin{lemma}\label{approachlem}
Let $t_0\geq 0$ be given, with $L(t_0)=L_0$. Then, for all $t\geq t_0$ one has
\begin{equation}\label{Lestimgen}
    L^2\leq L_0^2\exp(-k_1K_M(t-t_0))
        +\frac12\varepsilon_{SSl}^2\varepsilon_{MM}^2\left(1-\exp(-k_1K_M(t-t_0))\right).
\end{equation}
In particular, with $t_0=0$ and   $L(0)=e_0s_0/(s_0+K_M)=s_0\varepsilon_{SSl}$ for 
the initial value $(s_0,0)$,
\begin{equation}\label{Lestim}
\begin{array}{rcl}
     \dfrac{L^2}{s_0^2}& \leq& \varepsilon_{SSl}^2\left(\exp(- k_1K_Mt) 
        + \frac12 \varepsilon_{MM}^2 (1-\exp(- k_1K_Mt))\right)\\
     & =&\varepsilon_{SSl}^2\left(\exp(- (k_{-1}+k_2)t) 
        + \frac12 \varepsilon_{MM}^2 (1-\exp( - (k_{-1}+k_2)t))\right).
\end{array}
\end{equation}
\end{lemma}

\begin{proof}
Compare \eqref{diffineqvar} with the differential equation
\begin{equation*}
\cfrac{dV}{dt}= -(k_{-1}+k_2)\cdot V
    +\frac12(k_{-1}+k_2)\varepsilon_{SSl}^2\varepsilon_{MM}^2\cdot s_0^2 
\end{equation*}
for $V:=L^2$. The explicit solution of this linear equation for $V$ and a differential 
inequality argument (or Gronwall) yield the asserted estimates.
\end{proof}

\begin{remark}\label{lsquarerem}{\em 
The first factor on the right hand side of \eqref{Lestim} is the square of the 
Segel--Slemrod parameter $\varepsilon_{SSl}$, and the factor inside the second bracket 
is the square of the local timescale parameter $\varepsilon_{MM}$. Since $L/s_0$ is 
generally bounded by $\varepsilon_{SSl}$, the relevant parameter for the distance to
the QSS manifold will be $\varepsilon_{MM}$. \textcolor{black}{This and the above holds 
for any choice of parameters.}

For a more detailed inspection, we assume $\varepsilon_{MM}<1$, so the right hand side
of \eqref{Lestim} decreases with $t$. Then, for all $t\geq \widehat t$,
\begin{equation}\label{toldcrit}
\widehat t:=\dfrac{2}{k_{-1}+k_2}\cdot 
    \log\left(\dfrac{(k_{-1}+k_2)^2}{k_1k_2e_0} \right)
        =\dfrac{2}{k_1K_M}\log \dfrac{1}{\varepsilon_{MM}}
\end{equation}
one obtains that
\begin{equation}\label{eqLest}
 \dfrac{|c-g(s)|}{s_0} \leq  \sqrt\frac32\cdot\varepsilon_{SSl}\cdot \varepsilon_{MM}.
\end{equation}
To verify the inequality, it suffices to do so for $t=\widehat t$, and this, in turn, 
follows from the fact that
 \begin{equation*}
\exp(-(k_{-1}+k_2)t)={\varepsilon_{MM}}^2
 \end{equation*}
is solved by $t=\widehat t$.

This provides a first estimate for the approach to the slow manifold.}
\end{remark}

\begin{remark}\label{someelserem} {\em 
One may consider similar estimates for complex QSS, with no a priori reference to 
singular perturbations, in the system with substrate inflow:
\begin{equation*}
\begin{array}{rclclcl}
\dot s&=&k_0 -k_1e_0s&+&(k_{-1}+k_1s)c   \\
\dot c&=& k_1e_0s&-&(k_{-1}+k_2+k_1s)c. \\
\end{array}
\end{equation*}
Here, it is appropriate to choose initial values $s(0)=c(0)=0$. The chain of inequalities 
above works similarly, with the crucial difference that $\dot s=k_0 -\dot c+k_2c$. So, 
the assumption $e_0=\varepsilon e_0^*$ will no longer result in an order $\varepsilon^4$ 
term in the analogue of\eqref{Lestim} (unless $k_0$ is also of order $\varepsilon$); 
only order $\varepsilon^2$ can be salvaged. For more details, please see the discussion 
in Eilertsen et al.~\cite[Subsection 4.4]{ERSW}.}
\end{remark}

For $t\leq t_{\rm cross}$ an alternative Lyapunov function approach is suggested by 
{\bf Lemma~\ref{schahelem}}. We start with a variant of equation~\eqref{quadlyap}:
\begin{equation}\label{linlyap}
    \cfrac{d}{dt} (c-g(s)) = -(k_{-1}+k_2+k_1s)(c-g(s)) - g'(s)\dot{s}.
\end{equation}
Again, let $L:=c-g(s)$. Similar to the derivation of {\bf Lemma~\ref{approachlem}}, we 
find for $t\leq t_{\rm cross}$ (using $c-g(s)\leq 0$):
\begin{subequations}
\begin{align}
    \cfrac{ d L}{d t} &= -(k_{-1}+k_2+k_1s))L -g'(s)\dot{s}\\
    &=-(k_{-1}+k_2+k_1s)L +g'(s)\left(\dot{c}+k_2c\right)\\
    &= -\left((1+g'(s))(k_{-1}+k_2+k_1s)\right)L+k_2g'(s)c\\
     &= -\left((1+g'(s))(k_{-1}+k_2+k_1s)\right)L+k_2g'(s)(L+g(s)).
\end{align}
\end{subequations}
So, we have 
\begin{lemma}\label{linestlem} 
Consider the solution of \eqref{eqmmirrev} with initial value $(s_0,0)$ at $t=0$. Then, 
for $0\leq t\leq t_{\rm cross}$,
\begin{equation}\label{linLest}
    \cfrac{ d L}{dt}= -A\,L + B
\end{equation}
with 
\[
\begin{array}{rcl}
  A :  & =&k_{-1}+k_2+k_1s+g'(s)(k_{-1}+k_1s),\\
  B :  & =&k_2g'(s)g(s).
\end{array}
\]
\end{lemma}
Note that when $g(s)=c$ and $s>0$, then ${dL}/{dt}= k_2g'(s) g(s)> 0$.

\subsection{Second Step: The crossing time}
We will use {\bf Lemma~\ref{linestlem}} to compute upper and lower bounds, $t_{u}$ 
and $t_{\ell}$, such that $t_{\rm cross}\in [t_{\ell},t_u]$. The strategy will be 
to extract $t_{u}$ and $t_{\ell}$ from appropriate differential inequalities. We 
will express most of our estimates via the Segel--Slemrod parameter~$\varepsilon_{SSl}$. 
Although --- as mentioned in the Introduction --- the parameters used by Briggs and 
Haldane, or by Reich and Selkov, would be equally applicable in any well-defined 
limit with $e_0\to 0$ (and all other parameters in a compact subset of the open 
positive orthant), the Segel--Slemrod parameter turns out to be the most convenient.

We first determine a lower bound $t_\ell$. By \eqref{classsmeq} and \eqref{gprime},
we obtain
\[
g(s)\leq\dfrac{e_0s_0}{K_M+s_0}\text{  and  } g'(s)\leq \dfrac{e_0}{K_M}
            \text{  for  }0\leq s\leq s_0.
\]
Now, with the notation of {\bf Lemma~\ref{linestlem}}, we have
\begin{equation*}
    \begin{array}{rcl}
       A  & =& k_{-1}+k_2+k_1s+\dfrac{e_0K_M}{(K_M+s)^2}(k_{-1}+k_1s) \\
         & \leq& k_1(K_M+s)\left(1+\dfrac{e_0K_M}{(K_M+s)^2}\right)\\
         &=&k_1(K_M+s)+k_1e_0\dfrac{K_M^2}{(K_M+s)^2}\\
         &\leq &k_1(K_M+e_0+s_0)=:A^*,
    \end{array}
    \end{equation*}
and furthermore\footnote{For $B^*$, we simply use $\max g\cdot\max g'$, both on 
$[0,\,s_0]$. The global maximum of $B$ on $[0,\,\infty)$ equals $4/27\cdot k_2e_0^2/K_M$. 
For the record, we point out that using this estimate would not make an essential 
difference.}
    \begin{equation*}
        B\leq k_2\cdot \dfrac{e_0s_0}{K_M+s_0}\cdot\dfrac{e_0}{K_M}=:B^*.
    \end{equation*}
Since for $t\leq t_{\rm cross}$ one has $L\leq 0$, the differential 
equation~\eqref{linLest} implies the inequality
\begin{equation}
    \dfrac{dL}{dt}\leq -A^*\,L+B^*.
\end{equation}
Thus, defining $L^*$ by
\begin{equation*}
    \dfrac{dL^*}{dt}=-A^*L^*+B^*,\quad L^*(0)=-g(s_0),
\end{equation*}
one obtains that $L(t)\leq L^*(t)$ for $0\leq t\leq t_{\rm cross}$. Explicitly,
\begin{equation}\label{lowerestlin}
\begin{array}{rcl}
    L^* & =&-\left(\dfrac{B^*}{A^*}+g(s_0)\right)\exp(-A^*t)+\dfrac{B^*}{A^*} \\
    \\
     & =&s_0\varepsilon_{SSl}\left(-\left(1+\dfrac{k_2}{k_1K_M(1+\varepsilon_{SSl})}
            \varepsilon_{SSl}\right)\,\exp{(-(1+\varepsilon_{SSl})\lambda t)}
            +\dfrac{k_2}{k_1K_M(1+\varepsilon_{SSl})}\varepsilon_{SSl}\right),
\end{array}
\end{equation}
where $\lambda:=k_1(K_M+s_0)$. Now define $t_\ell$ by $L^*(t_\ell)=0$. A 
straightforward calculation shows
\begin{equation}\label{telldef}
\begin{array}{rcl}
t_\ell&=&\dfrac{1}{k_1(K_M+s_0)(1+\varepsilon_{SSl})}\,
    \log\left(1+\dfrac{k_{-1}+k_2}{k_2}(1+\varepsilon_{SSl})\cdot\dfrac{1}{\varepsilon_{SSl}}\right)\\
     & =&\dfrac{1}{(k_1s_0+k_{-1}+k_2)(1+\varepsilon_{SSl})}\,
     \log\left(1+\dfrac{k_{-1}+k_2}{k_2}(1+\varepsilon_{SSl})\cdot\dfrac{1}{\varepsilon_{SSl}}\right)
\end{array}
\end{equation}
This provides a lower estimate:

\begin{lemma}\label{lowertimelem}
For the solution of \eqref{eqmmirrev}, with initial value $(s_0,0)$, one has
$t_{\rm cross}\geq t_\ell$.
\end{lemma}

\begin{proof}
Assume that $t_{\rm cross}<t_\ell$, then $L(t_\ell)>0$. This is a contradiction to 
$L(t_\ell)\leq L^*(t_\ell)=0$.
\end{proof}

Recall that Segel and Slemrod~\cite{SSl} introduced
\begin{equation}\label{SSltime}
    t_{SSl}:=\dfrac{1}{k_1(K_M+s_0)}=\dfrac{1}{k_{-1}+k_2+k_1s_0}
\end{equation}
to estimate the duration of the fast transient. This defines the appropriate timescale at 
the very start, but as we show below, it cannot reflect the full transient phase.

There is a slightly simplified estimate for $t_\ell$:

\begin{equation*}
    \begin{array}{rcl}
     t_\ell    & =&t_{SSl}\dfrac{1}{1+\varepsilon_{SSl}}\log\left(\dfrac{1}{\varepsilon_{SSl}} 
        \bigg[\dfrac{k_1K_M}{k_2}+\varepsilon_{SSl}\left(\dfrac{k_1K_M}{k_2}+1\right)\bigg]\right)\\
         & \geq&t_{SSl}\dfrac{1}{1+\varepsilon_{SSl}}\log\left(\dfrac{1}{\varepsilon_{SSl}} 
            \dfrac{k_1K_M}{k_2}\right)\\
         &=&t_{SSl}\dfrac{1}{1+\varepsilon_{SSl}}\log\left(\dfrac{1}{\varepsilon_{SSl}}\right)
            +\log\left( \dfrac{k_1K_M}{k_2}\right)\\
         &\geq & t_{SSl}\left(1-\varepsilon_{SSl}\right)\left(\log\left(\dfrac{1}{\varepsilon_{SSl}}\right)
            +\log\left( \dfrac{k_1K_M}{k_2}\right)\right).\\
    \end{array}
\end{equation*}
Therefore, we may define
\begin{equation}\label{tlowdag}
    t_\ell^\dagger:=t_{SSl}(1-\varepsilon_{SSl})\left(\log\left(\dfrac{1}{\varepsilon_{SSl}}\right)
        +\log\left( \dfrac{k_{-1}+k_2}{k_2}\right)\right)
\end{equation}
as a lower estimate for the crossing time. \textcolor{black}{In the limiting case $e_0\to 0$, an} 
asymptotic expansion of the right-hand side yields
\begin{equation}\label{tlowdagas}
    t_\ell^\dagger\sim t_{SSl}\left[ \log\dfrac{1}{\varepsilon_{SSl}}+\log\dfrac{k_{-1}+k_2}{k_2}+o(1)
    \right]
\end{equation}

For the slow timescale, chosen (in consistency with the choice of the small parameter) as 
$\tau=\varepsilon_{SSl}t$, the above observations yield a lower estimate with leading term of order 
$\varepsilon_{SSl}\cdot\log(1/\varepsilon_{SSl})$ in the asymptotic expansion. 

\begin{remark}\label{lcuberem}{\em 
At this point, it seems appropriate to reconsider the notion ``onset of slow dynamics'' 
\textcolor{black}{for the case of small $\varepsilon_{SSl}$}. For 
system~\eqref{eqmmirrev}, we noted that the distinguished time $t_{\rm cross}$ 
{\rm (see, {\bf Lemma~\ref{schahelem}} and the following ones)} is a natural choice from a 
biochemical perspective. But singular perturbation theory does not provide a precisely 
defined time for the onset of the slow phase. The following two observations are based on 
a fundamental criterion for slow dynamics, namely closeness of the solution to the 
QSS manifold:
\begin{itemize}
\item[i.] Equation~\eqref{lowerestlin} shows that $|L^*(t_{SSl})/s_0|\approx 
\varepsilon_{SSl}\exp(-1)$. But, since $|L/s_0|$ can always be estimated above by terms of 
order $\varepsilon_{SSl}$ {\rm [see, \eqref{Lestim}]}, this inequality does not indicate 
closeness to the QSS manifold. Thus, the onset of slow dynamics cannot be assumed near 
$t_{SSl}$, and the Segel--Slemrod time $t_{SSl}$ seriously underestimates the duration 
of the transient phase.
\item[ii.] One may replace the condition $L^*(t)=0$ from \eqref{lowerestlin} by an \textcolor{black}{order 
$\varepsilon_{SSl}^2$ closeness condition}, requiring $L^*(t)/s_0 \geq -M\cdot \varepsilon_{SSl}^2$,
with some positive constant $M$, as the defining characteristic of the slow phase. A 
provisional definition of $t_{ons}$ by $L^*(t_{ons})/s_0 = -M\cdot \varepsilon_{SSl}^2$ 
yields $L^*(t)/s_0\geq -M\cdot \varepsilon_{SSl}^2$ for $t_{ons}\leq t\leq t_{\rm cross}$. 
Similar to the derivation of \eqref{telldef}, one obtains an estimate
\begin{equation}\label{tonsdef}
    t_{ons}=t_{SSl}\log(M^*/\varepsilon_{SSl})+\cdots
\end{equation}
with some constant $M^*$, and the dots representing higher order terms. Thus, we have the 
same lowest order asymptotic term $\log(1/\varepsilon_{SSl})$ as for $t_\ell^\dagger$.
\end{itemize}
}
\end{remark}

We proceed to estimate initial substrate depletion:
\begin{proposition}\label{initlosslowprop}
One has the inequality
\begin{equation*}
    \dfrac{s(t_\ell^\dagger)}{s_0}\leq \dfrac{k_{-1}}{k_{-1}+k_2}
        +\dfrac{k_{2}}{k_{-1}+k_2}\exp\left(-\dfrac{\varepsilon_{SSl}K_M}{(K_M+s_0)}(1-\varepsilon_{SSl})
        \cdot \log\left(\dfrac{k_1K_M}{\varepsilon_{SSl}k_2}\right)\right).
\end{equation*}
Moreover, when
\begin{equation}\label{lowertimecond}
    \varepsilon_{SSl}\cdot \log\left(\dfrac{k_1K_M}{k_2\varepsilon_{SSl}}\right)<1,
\end{equation}
then
\begin{equation}\label{lowscr}
  \dfrac{s_0-s_{\rm cross}}{s_0}\geq \dfrac{s_0-s(t_\ell^\dagger)}{s_0}\geq
        \dfrac{k_2}{2k_1(K_M+s_0)}\varepsilon_{SSl}(1-\varepsilon_{SSl})
            \log\left(\dfrac{k_1K_M}{\varepsilon_{SSl}k_2}\right).
\end{equation}
\end{proposition}

\begin{proof}
The estimate for $s(t_\ell^\dagger)$ is obtained by substitution of $t_\ell^\dagger$ in 
\eqref{notsosimplesest}. As for \eqref{lowscr}, the first  inequality holds because $s$ is 
decreasing with $t$. Then one directly obtains
\begin{equation*}
    \dfrac{s_0-s(t_\ell^\dagger)}{s_0}\geq \dfrac{k_{2}}{k_{-1}+k_2}\left(1-\exp(-\alpha)\right)
\end{equation*}
with
\begin{equation*}
    \alpha:=\dfrac{\varepsilon_{SSl}K_M}{K_M+s_0}(1-\varepsilon_{SSl})\cdot 
            \log\left(\dfrac{k_1K_M}{\varepsilon_{SSl}k_2}\right)<\varepsilon_{SSl}\cdot 
                \log\left(\dfrac{k_1K_M}{\varepsilon_{SSl}k_2}\right).
\end{equation*}
Condition~\eqref{lowertimecond} implies that $\alpha<1$, and the exponential series and 
the Leibniz criterion show that
\begin{equation*}
    \exp(-\alpha)\leq 1-\alpha+\alpha^2/2\leq 1-\alpha/2.
\end{equation*}
This estimate yields the second assertion.
\end{proof}

\begin{remark}\label{badsellrem}{\em 
The estimate in {\bf Proposition~\ref{initlosslowprop}} can be improved, subject to more restrictive 
assumptions on $\varepsilon_{SSl}$. Replacing~\eqref{lowertimecond} by
\begin{equation}\label{lowertimecondsharp}
    \varepsilon_{SSl}\cdot \log\left(\dfrac{k_1K_M}{\varepsilon_{SSl}k_2}\right)<r
\end{equation}
for some $r,\,0<r\leq 1$, it is straightforward to see that 
\begin{equation}\label{lowscrsharp}
\dfrac{s_0-s_{\rm cross}}{s_0}\geq(1-r/2)\,\dfrac{k_2}{k_1}
    \dfrac{\varepsilon_{SSl}}{K_M+s_0}(1-\varepsilon_{SSl})
        \log\left(\dfrac{k_1K_M}{\varepsilon_{SSl}k_2}\right)
\end{equation}
in this case. This suggests a simplified asymptotic estimate for $s_{\rm cross}$ by setting, 
for instance, $r=\sqrt{\varepsilon_{SSl}}$ and keeping only lowest order terms,
\begin{equation*}
    \dfrac{s_0-s_{\rm cross}}{s_0}\geq\dfrac{k_2}{k_1(K_M+s_0)}\varepsilon_{SSl}
        \left[\log\left(\dfrac{1}{\varepsilon_{SSl}}\right)+\log\left(\dfrac{k_1K_M}{k_2}\right)\right]+\cdots
\end{equation*}
In comparison, the asymptotic expansion of the right-hand side of \eqref{lowscr} starts with
\begin{equation}\label{lowscras}
  \dfrac{k_2}{2k_1(K_M+s_0)}\varepsilon_{SSl}(1-\varepsilon_{SSl})
    \log\left(\dfrac{k_1K_M}{\varepsilon_{SSl}k_2}\right)\sim \dfrac{k_2}{2k_1(K_M+s_0)}
        \varepsilon_{SSl}\left[\log\dfrac{1}{\varepsilon_{SSl}}+\log\dfrac{k_1K_M}{k_2}
  \right]+\cdots
\end{equation}
It turns out below that the (removable) factor $\cfrac12$ is less problematic for estimates 
than the factor $\cfrac{k_2}{k_1(K_M+s_0)}$.}
\end{remark}

\begin{remark}\label{nossltime}{\em 
\textcolor{black}{Thus, in the case of small $\varepsilon_{SSl}$,} for $(s_0-s_{\rm cross})/s_0$ 
one has a lower estimate by an expression asymptotic to
$\varepsilon_{SSl}\log(1/\varepsilon_{SSl})$. \textcolor{black}{Notably, this estimate indicates 
that the relative substrate depletion at crossing time is not of order $\varepsilon_{SSl}$. The widely 
held assumption in the literature (see, e.g. Segel \& Slemrod~\cite{SSl}) about negligibility 
of the substrate depletion in the pre-QSS phase should be seen from this perspective.}

Moreover, upon replacing $t_{\rm cross}$ by a differently chosen onset time $t_{ons}$ as 
in \eqref{tonsdef}, the argument in the proof of the Proposition, with 
{\bf Lemma~\ref{simplesestlem}} shows that
\begin{equation}
\dfrac{s_0-s(t_{ons})}{s_0}\geq\dfrac{k_2}{2k_1(K_M+s_0)}\varepsilon_{SSl}
    \log\left(\dfrac{M^*}{\varepsilon_{SSl}}\right)+\cdots
\end{equation}
whenever $\varepsilon_{SSl}\log\left(\dfrac{M^*}{\varepsilon_{SSl}}\right)<1$. Thus, the lowest 
order of the asymptotic expansion remains unchanged.}
\end{remark}

We now turn to upper bounds for $t_{\rm cross}$. \textcolor{black}{For technical reasons, since our 
argument requires a positive lower estimate for $g(s)$, we fix an auxiliary constant $0<q<1$ and 
consider the interval $[q\,s_0,\,s_0]$.} Then,
\begin{equation*}
   g(s)\geq \dfrac{qe_0s_0}{K_M+qs_0}\text{  and  } g'(s)\geq \dfrac{e_0 K_M}{(K_M+s_0)^2} 
   \quad\text{for}\quad qs_0\leq s\leq s_0.
\end{equation*}
Therefore,
\begin{equation*}
    A\geq k_1(qs_0+K_M)=:A_*,
\end{equation*}
and 
\begin{equation*}
    B\geq k_2\dfrac{qe_0s_0}{K_M+qs_0}\cdot\dfrac{e_0 K_M}{(K_M+s_0)^2}=:B_*
\end{equation*}
when $qs_0\leq s\leq s_0$.

Hence, for $0\leq t\leq t_{\rm cross}$ and $s(t)\geq qs_0$, one has
\begin{equation}
    \dfrac{dL}{dt}\geq -A_*L+B_*,
\end{equation}
and defining $L_*$ by 
\begin{equation*}
    \dfrac{dL_*}{dt}=-A_*L_*+B_*,\quad L_*(0)=-g(s_0),
\end{equation*}
the usual differential inequality argument shows $L\geq L_*$. Explicitly,
\begin{equation*}
    L_*=-\left(\dfrac{B_*}{A_*}+g(s_0)\right)\exp(-A_*t)+ \dfrac{B_*}{A_*}.
\end{equation*}
Define $t_u=t_u(q)$ by $L_*(t_u(q))=0$, thus
\begin{equation}\label{tupstep0}
    t_u(q)=\dfrac{1}{k_1(K_M+qs_0)}\,\log\left(1+C(q)\cdot\dfrac{1}{\varepsilon_{SSl}}\right),
    \quad \text{with} \quad C=C(q):=\dfrac{1}{q}\cdot\dfrac{(k_{-1}+k_2+qk_1s_0)^2}{k_2(k_{-1}+k_2)}.
\end{equation}
With the inequality
\begin{equation}\label{tupsteppoint5}
\dfrac1q<C(q)<\dfrac{C^*}{q},\quad C^*:=C(1)
    =\dfrac{(k_{-1}+k_2+k_1s_0)^2}{k_2(k_{-1}+k_2)}=\dfrac{k_1(K_M+s_0)^2}{k_2K_M},
\end{equation}
we obtain a more convenient estimate for $t_u(q)$:
\begin{equation*}
    \begin{array}{rcl}
      t_u(q)&=   & t_{SSl}\dfrac{K_M+s_0}{K_M+qs_0}\log\left(1+\dfrac{C(q)}{\varepsilon_{SSl}}\right) \\
         & \leq& t_{SSl}\dfrac{1}{q}\,\log\left(1+\dfrac{1}{\varepsilon_{SSl}}\dfrac{C^*}{q}\right).
    \end{array}
\end{equation*}
This gives rise to the upper estimate
\begin{equation}\label{tupdagger}
    t_u^\dagger(q):=t_{SSl}\dfrac{1}{q}\,\log\left(1+\dfrac{1}{\varepsilon_{SSl}}\dfrac{C^*}{q}\right)\geq t_u(q).
\end{equation}
For later use, we note
\begin{equation*}
\log\left(1+\dfrac{1}{\varepsilon_{SSl}}\dfrac{C^*}{q}\right)=\dfrac1q 
    \log\dfrac{1}{\varepsilon_{SSl}}+\log\dfrac{C^*}{q}+\log\left(1+\dfrac{q\varepsilon_{SSl}}{C^*}\right)
\end{equation*}
and obtain \textcolor{black}{in the limit $\varepsilon_{SSl}\to 0$} the asymptotic expansion
\begin{equation}\label{tupdaggeras}
t_u^\dagger(q)= t_{SSl}\dfrac{1}{q}\,\log\left(1+\dfrac{1}{\varepsilon_{SSl}}\dfrac{C^*}{q}\right)
    \sim \frac1qt_{SSl}\left[\log\dfrac{1}{\varepsilon_{SSl}}+\log\dfrac{C^*}{q}+o(1)\right].
\end{equation}

Equation~\eqref{tupstep0} provides an upper estimate for the crossing time, subject to an 
additional condition:
\begin{lemma}\label{uppertimelem}
Given $0<q<1$, assume that the solution of \eqref{eqmmirrev}, with initial value $(s_0,0)$, 
satisfies $s(t_u(q))\geq qs_0$. Then, $t_{\rm cross}\leq t_u(q)\leq t_u^\dagger(q)$.
\end{lemma}

\begin{proof}
Assume that $t_{\rm cross}>t_u(q)$, then $L_*(t_{\rm cross})>0$ and consequently 
$L(t_{\rm cross})>0$; a contradiction.
\end{proof}
\textcolor{black}{As will be seen below, the condition imposed in {\bf Lemma~\ref{uppertimelem}} will 
imply restrictions on $\varepsilon_{SSl}$.}

Modulo the hypothesis of {\bf Lemma~\ref{uppertimelem}}, we get an upper estimate for $t_{\rm cross}$ 
which is asymptotic to $\log(1/\varepsilon_{SSl})$, and complements the lower estimate $t_\ell$ 
with the same asymptotics. This clarifies the asymptotic behavior of $t_{\rm cross}$ as 
$\varepsilon_{SSl}\to 0$.

Still, criteria are needed to satisfy the hypothesis of the Lemma. The first step to obtain such 
criteria is to apply {\bf Lemma \ref{simplesestlem}} for $t=t_u^\dagger(q)$. By straightforward 
calculations, one finds the first estimate in the following proposition:
\begin{proposition}\label{stuplem}
One has
\begin{equation}\label{stuqupper}
\dfrac{s(t_u^\dagger(q))}{s_0}\geq \exp\left(-\dfrac{1}{q}\varepsilon_{SSl}
    \log\left(1+\dfrac{1}{\varepsilon_{SSl}}\cdot\dfrac{C^*}{q}\right)\right).
\end{equation}
Moreover, when 
\begin{equation*}
\varepsilon_{SSl}\log\left(1+\dfrac{1}{\varepsilon_{SSl}}\cdot\dfrac{C^*}{q}\right)<q
\end{equation*}
then
\begin{equation}\label{redinitvalest}
\dfrac{s_0-s_{\rm cross}}{s_0}\leq\dfrac{s_0-s(t_u^\dagger(q))}{s_0}
    \leq \dfrac{1}{q}\varepsilon_{SSl}\log\left(1+\dfrac{C^*}{q}\cdot\dfrac{1}{\varepsilon_{SSl}}\right).
\end{equation}
\end{proposition}

\begin{proof}
There remains estimate \eqref{redinitvalest}. The first inequality follows from monotonicity 
of $t\mapsto s(t)$. When the stated condition holds then 
\begin{equation*}
    \gamma:=\dfrac{1}{q} \varepsilon_{SSl}\log\left(1+\dfrac{1}{\varepsilon_{SSl}}
        \cdot\dfrac{C^*}{q}\right)<1,
\end{equation*}
and therefore $\exp(-\gamma)>1-\gamma$ by the exponential series and the Leibniz criterion. 
Substitution yields the assertion.
\end{proof}

Analogous to the derivation of expansion \eqref{tupdaggeras} \textcolor{black}{in the limiting 
case $\varepsilon_{SSl}\to 0$} one obtains an expansion of the 
right-hand side of equation \eqref{redinitvalest}, up to terms of order $o(\varepsilon_{SSl})$:
\begin{equation}\label{redinitvalestas}
    \dfrac{1}{q}\varepsilon_{SSl}\log\left(1+\dfrac{1}{\varepsilon_{SSl}}
        \cdot\dfrac{C^*}{q}\right)\sim \dfrac1q\left[\varepsilon_{SSl}
            \log\dfrac{1}{\varepsilon_{SSl}}+\varepsilon_{SSl}\log\dfrac{C^*}{q}+\cdots
    \right].
\end{equation}
 Equation~\eqref{stuqupper}, in view of $\lim_{x\to 0+}x\log(1/x)=0$, shows that for any fixed 
 $q$ the condition $s(t_u^\dagger(q))\geq qs_0$ holds for sufficiently small $\varepsilon_{SSl}$.
 
There remains to determine usable explicit bounds for $\varepsilon_{SSl}$ for given $q$.  We 
aim here at providing simple workable, rather than optimal, conditions:
\begin{proposition}\label{uppertimeprop} 
Let $q\geq \frac12$, such that
\begin{equation*}
    4q\log(1/q)\cdot\log(4C^*)<1. 
\end{equation*}
 Assume that
 \begin{equation*}
     \varepsilon_{SSl}< \exp(-1) \quad {\rm and}\quad 
        \varepsilon_{SSl}\leq\frac{9}{16}\left(q\log(1/q)\right)^2.
 \end{equation*}
Then, $s(t_u^\dagger(q))\geq q s_0$ and consequently $t_{\rm cross}\leq t_u^\dagger(q)$.
\end{proposition}

\begin{proof}
By {\bf Lemma~\ref{uppertimelem}}, it is sufficient to prove the inequality 
$s(t_u^\dagger(q))\geq qs_0$. By \eqref{stuqupper}, this holds whenever
\begin{equation*}
    \exp\left(-\dfrac{1}{q}\varepsilon_{SSl}\log\left(1+\dfrac{C^*}{q}\cdot
        \dfrac{1}{\varepsilon_{SSl}}\right)\right)\geq q.
\end{equation*}
Equivalently
\begin{equation*}
    \dfrac{1}{q}\varepsilon_{SSl}\log\left(1+\dfrac{C^*}{q}
        \cdot\dfrac{1}{\varepsilon_{SSl}}\right)\leq \log\dfrac1q,
\end{equation*}
or
\begin{equation}\label{tustep2point5}
    \varepsilon_{SSl}\log\left(1+\dfrac{C^*}{q}
        \cdot\dfrac{1}{\varepsilon_{SSl}}\right)\leq q\log\dfrac1q.
\end{equation}
Rewrite the left hand side as
\begin{equation*}
\begin{array}{rcl}
\varepsilon_{SSl}\log\left(1+\dfrac{C^*}{q}\cdot\dfrac{1}{\varepsilon_{SSl}}\right) &=&
    \varepsilon_{SSl}\log\left(\dfrac{1}{\varepsilon_{SSl}}\left(\varepsilon_{SSl}
        +\dfrac{C^*}{q}\right)\right) \\
    &=&\varepsilon_{SSl}\log\left(\dfrac{1}
     {\varepsilon_{SSl}}\right)+\varepsilon_{SSl}\log\left(\varepsilon_{SSl}+\dfrac{C^*}{q}\right)\\
    &\leq&\varepsilon_{SSl}\log\left(\dfrac{1}
      {\varepsilon_{SSl}}\right)+\varepsilon_{SSl}\log\left(2\cdot\dfrac{C^*}{q}\right)\\
    &\leq&\varepsilon_{SSl}\log\left(\dfrac{1}{\varepsilon_{SSl}}\right)+\varepsilon_{SSl}\log(4C^*),\\
\end{array}
\end{equation*}
where we used $C^*\geq1>\varepsilon_{SSl}$ and $\frac12\leq q<1$. In view of 
$\varepsilon_{SSl}\log(1/\varepsilon_{SSl})\leq \sqrt{\varepsilon_{SSl}}$, the 
inequality~\eqref{tustep2point5} holds whenever
\begin{equation}\label{tustep3}
   \sqrt{\varepsilon_{SSl}} +\varepsilon_{SSl}\log(4 C^*)
    \leq q\log\dfrac{1}{q}.
\end{equation}
For the remainder of this proof, we abbreviate $A:=\log(4C^*)$ and $B:= q\log\dfrac{1}{q}$. 
Let $\theta$ be the positive number with $A\theta^2+\theta=B$. Then, for any 
$\varepsilon_{SSl}\leq \theta^2$, the inequality \eqref{tustep3} holds. Now, the solution of 
the above quadratic equation with $AB<1/4$, Taylor expansion and the Leibniz criterion show
\begin{equation*}
    \theta=\frac1{2A}\left(-1+\sqrt{1+4AB}\right)\geq\frac1{2A}
        \left(-1+ 1+\frac{4AB}{2}-\frac{16A^2B^2}{8}\right),
\end{equation*}
hence
\begin{equation*}
       \delta\geq B(1-AB)\geq \frac34 B.
\end{equation*}
Thus, inequality \eqref{tustep3} holds whenever $\varepsilon_{SSl}\leq (\frac34 B)^2$.
\end{proof}

The role of the constant $q$ is mostly auxiliary. It serves to ensure the applicability of 
{\bf Proposition~\ref{uppertimeprop}}, but actual estimates e.g.\ of $s(t_u^\dagger(q))$ will 
rely on {\bf Proposition~\ref{stuplem}}. 

\begin{example}{\em 
We consider one particular setting for the purpose of illustration. Assume that 
\begin{equation*}
    C^*=\dfrac{(k_{-1}+k_2+k_1s_0)^2}{k_2(k_{-1}+k_2)}\leq 250.
\end{equation*}
This condition covers a wide range of reaction parameters, for instance it is satisfied 
whenever $s_0\leq 5 K_M$ and $k_{-1}\leq 5 k_2$. Then, the requirement on $q$ in 
{\bf Proposition~\ref{uppertimeprop}} is satisfied whenever $q\geq 0.97$. For $q=0.97$.
one finds the condition $\varepsilon_{SSl}\leq 4.9\cdot 10^{-4}$.}
\end{example}

Rather than $t_u^\dagger(q)$, one may consider a slightly weaker, but more convenient estimate. 
Fix $\varepsilon_{SSl}$ such that $s(t_u^\dagger(q))\geq qs_0$. We will prove that the relative 
error upon replacing $t_u^\dagger(q)$ by
\begin{equation}\label{EST}
    t_u^\dagger(1)=\dfrac{1}{k_1(K_M+s_0)}\log\left(1+\dfrac{k_1(K_M+s_0)^3}{k_2K_Me_0}\right)
        =t_{SSl}\log\left(1+\dfrac{C^*}{\varepsilon_{SSl}}\right)
\end{equation}
is approximately equal to $(1-q)$ when $q$ approaches $1$.

\begin{lemma}\label{tuto1}
One has
\begin{equation*}
0\leq \dfrac{t_u^\dagger(q)-t_u^\dagger(1)}{t_u^\dagger(q)}\leq \dfrac{1-q}{q}\cdot
        \dfrac{1+\log(1+C^*/(q\,\varepsilon_{SSl}))}{\log(1+C^*/(q\,\varepsilon_{SSl}))}.
\end{equation*}
\end{lemma}

\begin{proof}
We abbreviate $A=C^*/\varepsilon_{SSl}$ and consider the function
\begin{equation*}
    q\mapsto f(q):=\frac1q\log(1+A/q),
\end{equation*}
noting $t_u^\dagger(q)=t_{SSl} f(q)$. The derivative
\begin{equation*}
    f'(q)=-\frac1{q^2}\left(\log(1+A/q)+\dfrac{A}{A+q}\right)
\end{equation*}
is an increasing function of $q$. By the mean value theorem, one has $f(q)-f(1)=(q-1)f'(q^*)$ 
for some $q^*$ between $q$ and $1$. Hence, by monotonicity and with $A/(A+q)<1$,
\begin{equation*}
    f(q)-f(1)\leq (1-q)\left| f'(q)\right|\leq \frac{1-q}{q^2}\left(\log(1+A/q)+1\right).
\end{equation*}
The assertion follows.
\end{proof}

We also note an asymptotic expansion \textcolor{black}{as $\varepsilon_{SSl}\to 0$}:
\begin{equation}\label{ESTasex}
\begin{array}{rcl}
t_u^\dagger(1)&\sim& t_{SSl}\left[ \log\dfrac{1}{\varepsilon_{SSl}}
    +\log\left(\dfrac{k_1K_M}{k_2}\left(\dfrac{K_M+s_0}{K_M}\right)^2\right)+o(1)
\right]\\
    & \sim& t_{SSl}\left[ \log\dfrac{1}{\varepsilon_{SSl}}
    +\log\left(\dfrac{k_{-1}+k_2}{k_2}\left(\dfrac{k_{-1}+k_2+k_1s_0}{k_{-1}+k_2}\right)^2\right)+o(1)
\right].
\end{array}
\end{equation}

The numerical simulations underlying {\sc Figure~\ref{FIG1}} illustrate that $t_u^\dagger(1)$ 
is a quite good approximation of the crossing time. 
\begin{figure}[htb!]
  \centering
    \includegraphics[width=8.0cm]{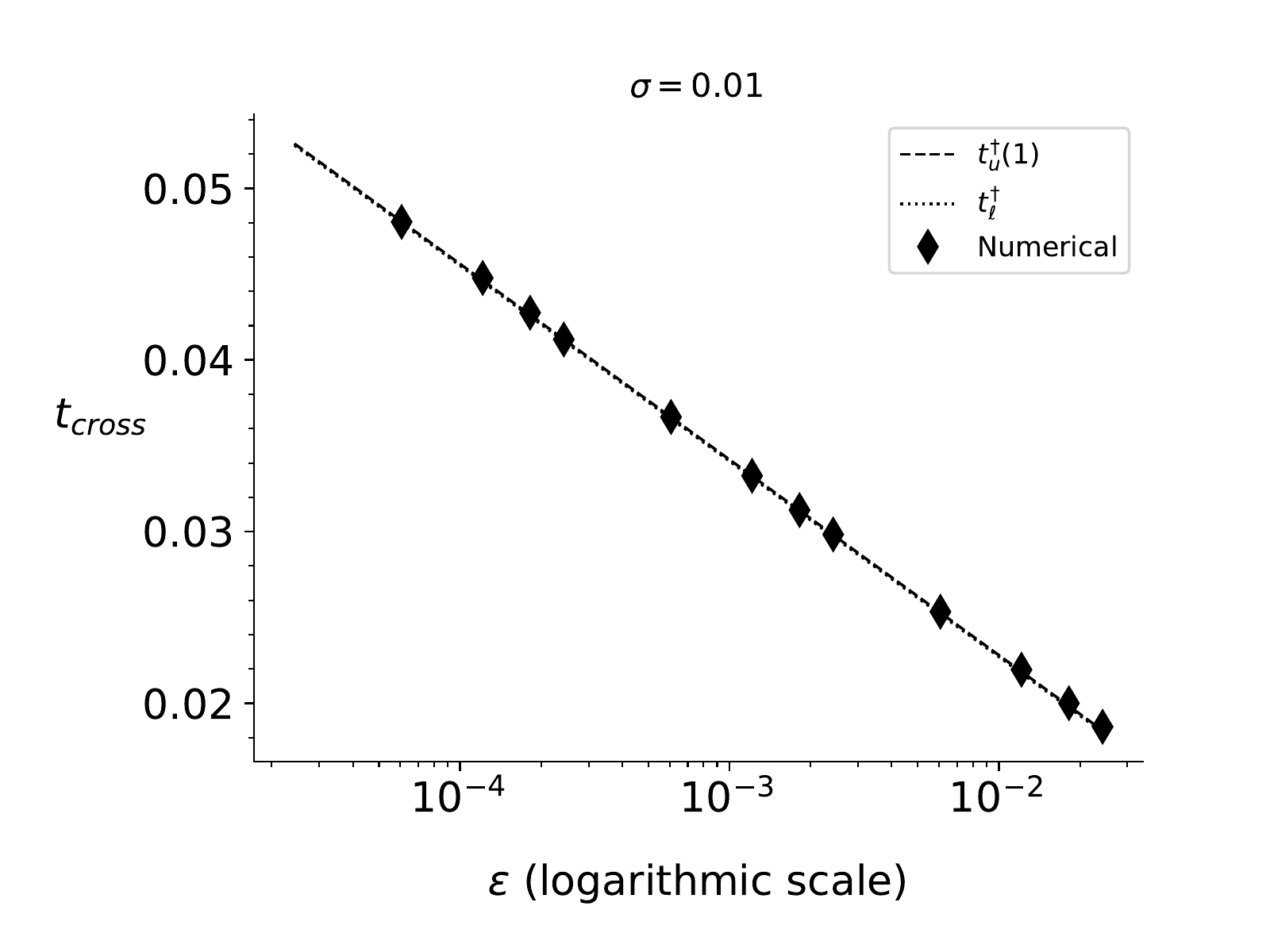}
    \includegraphics[width=8.0cm]{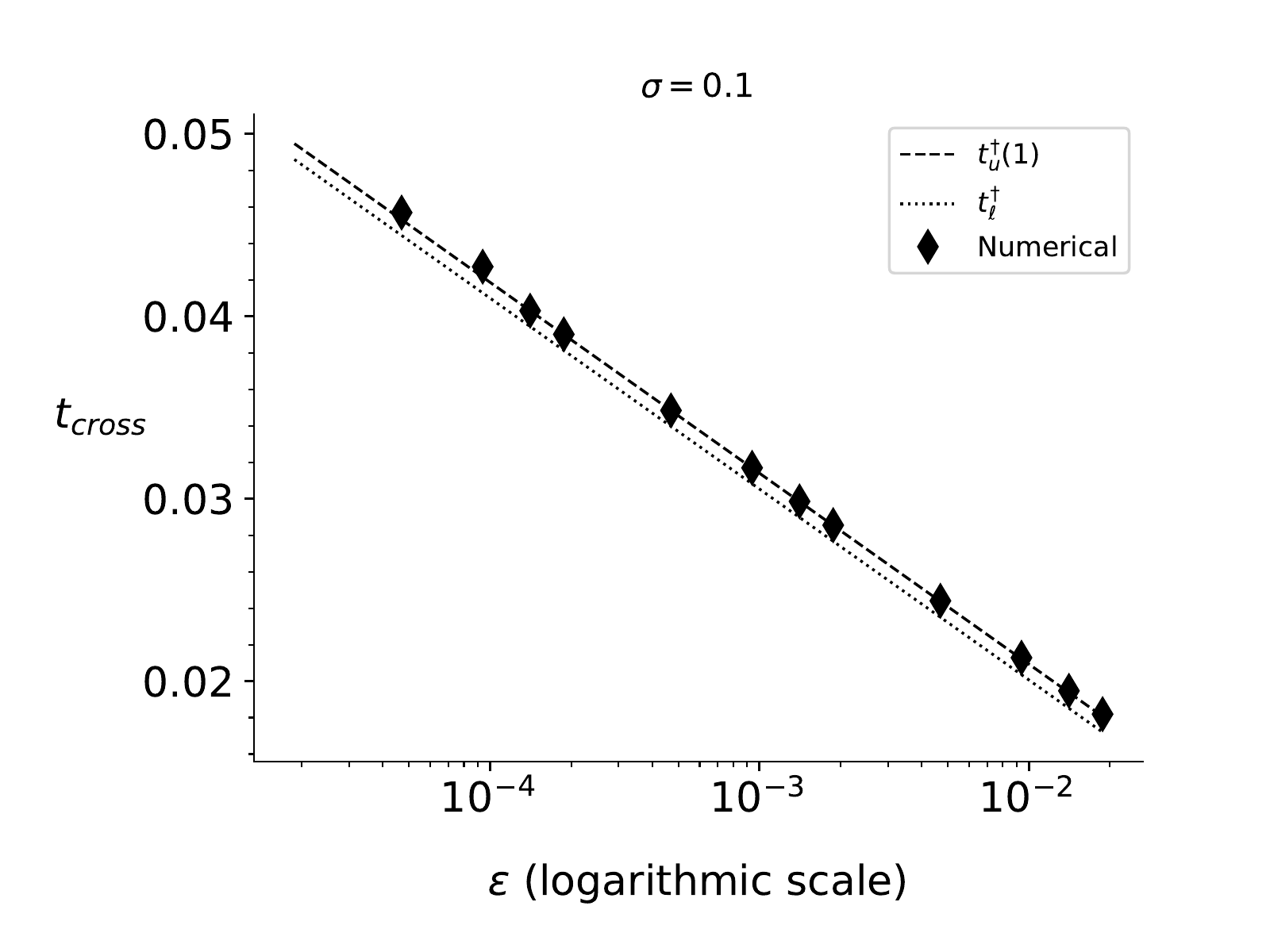}\\
    \includegraphics[width=8.0cm]{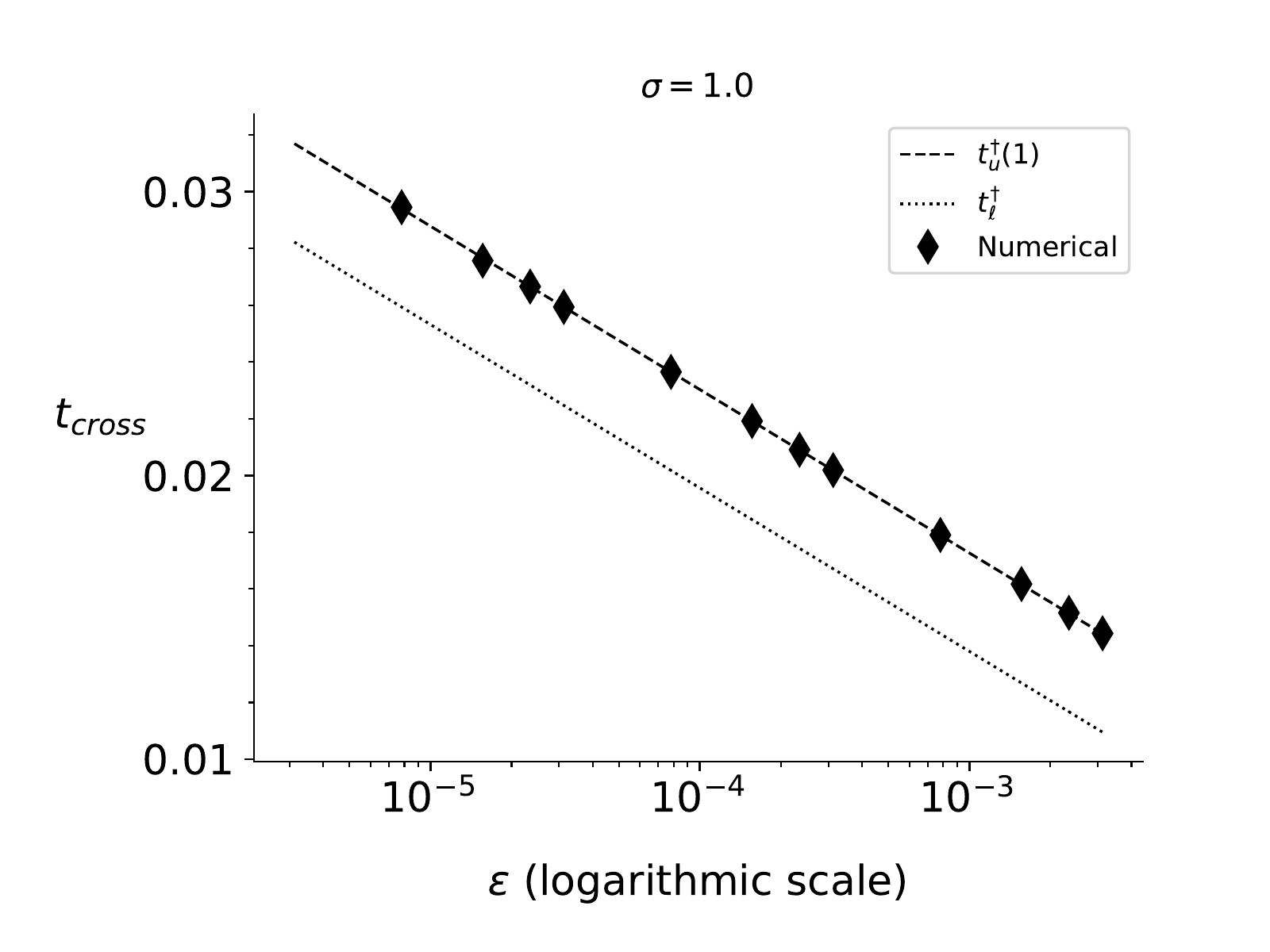}
    \includegraphics[width=8.0cm]{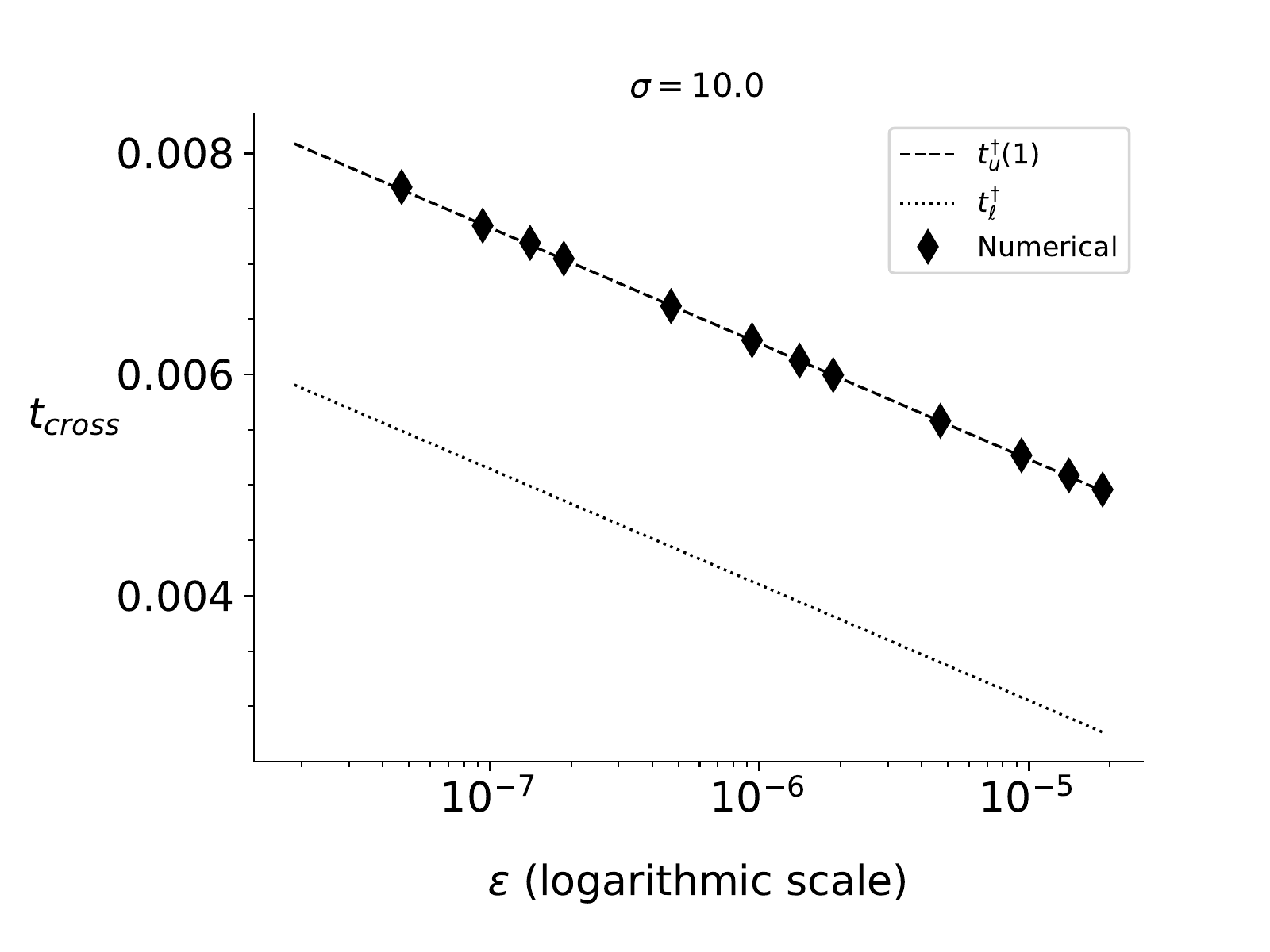}
\caption{\textbf{Numerical simulations indicate that $t_u^\dagger(1)$, defined in \eqref{EST}, 
is a reasonable estimation of $t_{\rm cross}$ when $e_0\ll K_M$.} In all panels, 
$e_0 \in [0.025, 0.05, 0.075, 0.1, 0.25, 0.5, 0.75, 1.0, 2.5, 5.0, 7.5, 10]$, 
$k_1=1.0$, $k_2=k_{-1}=100$, thus $K_M=200$, and $\sigma=s_0/K_M$. 
\textcolor{black}{The parameter $\varepsilon$ corresponds to the Reich-Selkov parameter 
$\varepsilon_{RS}=e_0/K_M$}. The solid black diamonds are 
the numerically estimated crossing times. The densely dashed line is obtained from (\ref{EST}). The dotted 
line is obtained from (\ref{tlowdagas}). {{\sc Top Left:}} $s_0 =2$. {{\sc Top Right:}} $s_0=20$. 
{{\sc Bottom Left:}} $s_0=200$. {{\sc Bottom Right:}} $s_0=2000$. Observe the noticeable 
difference between (\ref{ESTasex}) and (\ref{tlowdagas}) when $s_0$ is much larger than $K_M$. 
This is due to the difference in the constant terms of the expansions. One also sees that the 
lower estimate $t_\ell^\dagger$ from \eqref{tlowdagas} is worse than the upper estimate; compare 
Remark \ref{badsestrem}.}\label{FIG1}
\end{figure}
  
\begin{remark}\label{timeandsest}{\em 
Observe that $t_u^\dagger(1)$ corresponds to the estimate $T_{\rm in}$ from Noethen 
\& Walcher~\cite[Lemma 4]{NoWa}, but with an additional factor $(K_M+s_0)/K_M$ nested 
inside the logarithm. The presence of this term is relevant: the solution slows down 
significantly -- especially in the  $c$-direction -- near the $c$-nullcline in regions where 
$K_M \ll s$. In these regions, the solution will travel nearly horizontally and below the 
QSS manifold for an extended period of time before finally crossing. Moreover, the vanishing 
of $K_M$ gives rise to a line of equilibrium points at $c=e_0$. In this limiting case, the 
crossing time $t_{\rm cross}$ tends to infinity for any trajectory for which 
$c(0)\neq e_0$. This fact is reflected by the term $(K_M+s_0)/K_M$ in the expression for 
$t_u^\dagger(1)$.}
\end{remark}

Finally, it may be appropriate to look at the substrate depletion during the transitory phase 
from a general perspective: As shown by equations~\eqref{tonsdef} and \eqref{tupdagger} 
(setting $q=1$), the onset time for the slow dynamics will in any case be of the type
\begin{equation}
    t^*_{ons}= t_{SSl}\log\dfrac{M}{\varepsilon_{SSl}}+\cdots
\end{equation}
with some positive constant $M$. With slight modifications of {\bf Propositions~\ref{initlosslowprop}} 
and {\bf \ref{stuplem}} one arrives at 
\begin{equation}\label{initlossrough}
   \widehat M_1\cdot\varepsilon_{SSl}\,\log\dfrac{1}{\varepsilon_{SSl}}+\cdots
        \leq \dfrac{s_0-s(t_{ons}^*)}{s_0}
        \leq\widehat M_2\cdot\varepsilon_{SSl}\,\log\dfrac{1}{\varepsilon_{SSl}}+\cdots,
\end{equation}
with suitable constants $\widehat M_i$. Thus\textcolor{black}{, as $\varepsilon_{SSl}\to 0$} we 
have the asymptotic order 
$\varepsilon_{SSl}\,\log(1/\varepsilon_{SSl})$ for the relative initial substrate depletion.

\subsection{Third Step: Error estimates for the approximation}
We now turn toward global error estimates for the reduction. As in the previous subsection, 
we will express most estimates in terms of the Segel--Slemrod parameter $\varepsilon_{SSl}$.

For $t\geq t_{\rm cross}$, we consider the familiar Michaelis--Menten equation, augmented by 
an error term. We start from
\begin{equation}\label{mmerrorterm}
    \dot s=  -k_1e_0s+ (k_{-1}+k_1s)g(s)+(k_{-1}+k_1s)(c-g(s)).
\end{equation}

\begin{lemma}\label{enclose}
For all $t\geq t_{\rm cross}$, the $s$ entry of the solution of \eqref{eqmmirrev} with initial 
value $(s_0,0)$ satisfies
\begin{equation}\label{mmbothest}
\begin{array}{rcccccl}
    \dot s &\geq &-\dfrac{k_1k_2e_0s}{k_{-1}+k_2+k_1s};&& &&\\
    \dot s &\leq &-\dfrac{k_1k_2e_0s}{k_{-1}+k_2+k_1s} &+& \frac{1}{\sqrt{2}}k_1e_0s_0\cdot  
    \bigg(\cfrac{k_{-1}+k_1s_0}{k_{-1}+k_2+k_1s_0}\bigg)\cdot \bigg(\cfrac{k_1k_2e_0}
    {(k_{-1}+k_2)^2}\bigg)&=:&U(s).\\
\end{array}
\end{equation}
\end{lemma}

\begin{proof} 
For the first inequality note that $c-g(s)\geq 0$ for $t\geq t_{\rm cross}$. For the second 
inequality, using \eqref{diffineqvar} with $L(t_{\rm cross})=0$, one obtains 
\begin{equation}
    \dfrac{L^2}{s_0^2}\leq \frac12\, \left(\varepsilon_{MM}\varepsilon_{SSl}\right)^2
    \cdot \left(1-\exp{[-(k_{-1}+k_2)(t-t_{\rm cross})]}\right)\leq \frac12\, 
    \left(\varepsilon_{MM}\varepsilon_{SSl}\right)^2
\end{equation}
for all $t\geq t_{\rm cross}$, thus one has $L/s_0\leq \frac{1}{\sqrt2}\varepsilon_{SSl}\varepsilon_{MM}$, 
with 
\begin{equation}\label{mmlowredest}
\begin{array}{rcccccl}
\dot s  &=&  -k_1e_0s+ (k_1s+k_{-1})g(s)&+&(k_{-1}+k_1s)(c-g(s))& & \\
        &\leq& -\dfrac{k_1k_2e_0s}{k_{-1}+k_2+k_1s} &+& \cfrac{1}{\sqrt2}
         \dfrac{k_{-1}+k_1s}{(k_{-1}+k_1s_0)}\dfrac{k_1e_0s_0(k_{-1}+k_1s_0)}
         {(k_{-1}+k_2+k_1s_0)}\cdot\dfrac{k_1k_2e_0}{(k_{-1}+k_2)^2}&&\\
         &\leq& -\dfrac{k_1k_2e_0s}{k_{-1}+k_2+k_1s} &+& \cfrac{1}{\sqrt2}k_1e_0s_0\cdot
         \dfrac{k_{-1}+k_1s_0}{(k_{-1}+k_2+k_1s_0)}\dfrac{k_1k_2e_0}{(k_{-1}+k_2)^2} &=& U(s),\\
\end{array}
\end{equation}
\end{proof}

Defining the equilibrium dissociation constant of enzyme-substrate complex as
\begin{equation}\label{KSdef}
    K_S:=\dfrac{k_{-1}}{k_1},
\end{equation}
one may rewrite
\begin{equation*}
U(s)=-\dfrac{k_2e_0s}{K_M+s}+\frac{1}{\sqrt 2}e_0s_0\left(\dfrac{K_S+s_0}{K_M+s_0}\right)\cdot\dfrac{k_2e_0}{K_M^2}.
\end{equation*}

\begin{remark}{\em 
By the same token, one obtains an estimate for product formation:
\begin{equation}\label{mmlowredest2}
\begin{array}{rcccccl}
\dot p &=  &  k_2g(s)&+&k_2(c-g(s))& & \\
     &\leq&\dfrac{k_1k_2e_0s}{k_{-1}+k_2+k_1s} &+& 
        \cfrac{1}{\sqrt2}\dfrac{k_1k_2e_0s_0}{(k_{-1}+k_2+k_1s_0)}\cdot\dfrac{k_1k_2e_0}{(k_{-1}+k_2)^2}&&\\
     &\leq&\dfrac{k_1k_2e_0s}{k_{-1}+k_2+k_1s} &+& 
        \cfrac{1}{\sqrt2}k_1e_0s_0\cdot\dfrac{k_{-1}+k_1s_0}{(k_{-1}+k_2+k_1s_0)}
            \cdot\dfrac{k_1k_2e_0}{(k_{-1}+k_2)^2}
    &=&\widetilde U(s).\\
\end{array}
\end{equation}
}
\end{remark}
For the following fix $\widetilde t\geq t_{\rm cross}$. By differential inequality arguments, we will 
estimate the difference of the $s$ entry of the solution of  \eqref{eqmmirrev} with initial values 
$(\widetilde s, \widetilde c)):=(s(\widetilde t), c( \widetilde t))$ -- which is just a time shift of 
the solution of \eqref{eqmmirrev} with initial values $(s_0,0)$ -- and the Michaelis--Menten equation 
with initial value $\widetilde s$. We will base our estimates on the auxiliary result below.

\begin{lemma}\label{slowestlem}
Let $a$, $b$ and $c$ be positive real numbers, $x_0,\,t_0\in\mathbb R$ and consider the initial 
value problems
\[
\begin{array}{rclcl}
\dot x&=&-\dfrac{cx}{x+a}+b&,&x(t_0)=x_0;\\
\dot y&=&-\dfrac{cy}{y+a} &,&y(t_0)=x_0;\\
\dot z&=& -\dfrac{cz}{z+a}&,& z(t_0)=z_0>x_0.
\end{array}
\]
Then,
\begin{enumerate}[(a)]
\item For all $t>t_0$, one has $x(t)-y(t)>0$.
\item Additionally, assume that $x_0>ab/(c-b)$. Then, $x(t)$ decreases for all $t>t_0$. We find 
that
\[
x(t)-y(t)\leq \dfrac{b\cdot(x_0+a)^2}{ac}\cdot\left(1-\exp\left[-\dfrac{ac}{(x_0+a)^2}\,(t-t_0)\right]\right)
    \leq \dfrac{b\cdot(x_0+a)^2}{ac}
\]
\item For all $t>t_0$, one has 
\[
0<z(t)-y(t)\leq (z_0-x_0)\cdot\exp\left[-\dfrac{ac}{(z_0+a)^2}\,(t-t_0)\right]\leq z_0-x_0.
\]
\end{enumerate}
\end{lemma}

\begin{proof} Part (a), due to $x(0)=y(0)$, follows directly from a standard result on differential 
inequalities.

Turning to the proof of part (b), note that $x^*=ab/(c-b)$ is the only stationary point of the 
differential equation for $x$. So, the solution with initial value $x_0>x^*$ is strictly decreasing 
and converges to this point. Now, we have
\[
\begin{array}{rcl}
\cfrac{d}{dt}\left(x-y\right)&=& -c\left(\dfrac{x}{x+a}-\dfrac{y}{y+a}\right)+b\\
 &=&-c\dfrac{a(x-y)}{(x+a)(y+a)} +b\\
 &\leq & -\dfrac{ac}{(x_0+a)^2}\,(x-y)+b.
\end{array}
\]
Compare this with the solution of the initial value problem
\[
\dot v=-\dfrac{ac}{(x_0+a)^2}v+b,\quad v(0)=0
\]
to obtain the assertion. As for part (c), the first inequality is immediate, while the second is 
verified by a variant of the previous argument, with the inequality
\[
\begin{array}{rcl}
\cfrac{d}{dt}\left(z-y\right)&=& -c\left(\dfrac{z}{z+a}-\dfrac{y}{y+a}\right)\\
 &\leq & -\dfrac{ac}{(z_0+a)^2}\,(z-y).
\end{array}
\]
\end{proof}
Evaluating the constant in part (b) of {\bf Lemma~\ref{slowestlem}} with $a=K_M$, 
$b= \frac{1}{\sqrt2}e_0s_0\left(\cfrac{k_2e_0}{K_M^2}\right)\cfrac{K_S+s_0}{K_M+s_0}$, $c=k_2e_0$ 
and $x_0=\widetilde s$, we obtain
\begin{subequations}
\begin{align*}
\frac{1}{\sqrt2}s_0\cdot \dfrac{e_0}{K_M}\cdot\dfrac{K_S+s_0}{K_M+s_0}\dfrac{K_M+\widetilde s}{K_M}\cdot \dfrac{K_M+\widetilde s}{K_M}
    &\leq\frac{1}{\sqrt2}s_0\cdot \dfrac{e_0}{K_M}\cdot\dfrac{K_S+s_0}{K_M+s_0}\left(\dfrac{K_M+s_0}{K_M}\right)^2\\
    &=\frac{1}{\sqrt2}s_0\cdot \varepsilon_{RS}\cdot\dfrac{(K_M+s_0)(K_S+s_0)}{K_M^2}.
\end{align*}
\end{subequations}
Choosing a natural scaling (and omitting the factor $\frac{1}{\sqrt 2}$), the parameter
\begin{equation}\label{penultisp}
\varepsilon_{L}:=\varepsilon_{RS}\dfrac{(K_M+s_0)(K_S+s_0)}{K_M^2}=\dfrac{e_0}{K_M}\dfrac{(K_M+s_0)(K_S+s_0)}{K_M^2}
    =\dfrac{k_1e_0}{k_{-1}+k_2}\dfrac{(k_{-1}+k_2+k_1s_0)(k_{-1}+k_1s_0)}{(k_{-1}+k_2)^2}
\end{equation}
provides an upper estimate for the long-term accuracy of the reduction. Note that the index indicates 
that the parameter was obtained from a linear differential inequality.

\subsubsection{Estimates for the slow dynamics: Special case}
In the application of the QSSA, it is generally assumed that there is an initial transient during 
which the substrate concentration remains approximately constant or changes slowly while the complex 
concentration builds up. This assumption -- that the substrate concentration does not change 
significantly during this initial transient -- is known as the reactant stationary 
approximation \cite{rsa,validity}. The general assumption is that $s\approx s_0$ from $t=0$ 
until $t_{\rm cross}$. However, this a qualitative estimate. A more careful analysis is required 
in order to formulate a quantitative assertion concerning the validity of the reactant stationary 
approximation.

We first determine estimates given the special 
assumption that the substrate concentration at the start of the slow phase is exactly known. In view 
of {\bf Lemma~\ref{enclose}} and {\bf Lemma~\ref{slowestlem}}, we then obtain
\begin{proposition}\label{penultimatepropone}
Denote by $s(t)$ the first component of the solution of \eqref{eqmmirrev} with initial value 
$(s_0,0)$ at $t=0$. Moreover, let $\widetilde t\geq t_{\rm cross}$, $\widetilde s:=s(\widetilde t)$ 
and define $\underline{s}$, resp.\ $\overline{s}$ by
\begin{equation}\label{mmbothestapp}
\begin{array}{rcccl}
\dot{\underline{s}} &= &-\dfrac{k_2e_0\underline{s}}{K_M+\underline{s}} &,& \underline{s}(\widetilde t)=\widetilde s;\\
\dot{\overline{s}}  & =&-\dfrac{k_1k_2e_0\overline{s}}{K_M+\overline{s}} + 
    \sqrt2k_1e_0s_0\cdot \bigg(\cfrac{k_{-1}+k_1s_0}{k_{-1}+k_2+k_1s_0}\bigg)\cdot \bigg(\cfrac{k_1k_2e_0}{(k_{-1}+k_2)^2}\bigg)
        &,& \overline{s}(\widetilde t)=\widetilde s.\\
\end{array}
\end{equation}
Then, for all $t\geq \widetilde t$, we have 
\begin{equation}\label{sandwich}
    \underline{s}(t)\leq s(t)\leq\overline{s}(t)
\end{equation}
and
\begin{equation}\label{lockup}
    \overline{s}(t)-s(t)\leq s_0\cdot\varepsilon_{L};\quad s(t)-\underline{s}(t)\leq s_0\cdot\varepsilon_{L}.
\end{equation}
\end{proposition}
\begin{proof}
To prove \eqref{sandwich}, use {\bf Lemma~\ref{enclose}}. Moreover, parts (a) and (b) of 
{\bf Lemma~\ref{slowestlem}} show that 
\[
\overline{s}(t)-\underline{s}(t)\leq \frac{1}{\sqrt2}s_0\varepsilon_{RS}\cdot 
    \dfrac{K_S+s_0}{K_M+s_0}\cdot\left(\dfrac{K_M+\widetilde s} {K_M}\right)^2 \leq 
    \frac{1}{\sqrt2}s_0\varepsilon_{RS}\cdot\dfrac{(K_M+s_0)(K_S+s_0)}{K_M^2}<s_0\,\varepsilon_{L},
\]
which in combination with \eqref{sandwich} proves \eqref{lockup}.
\end{proof}
There is a different approach to upper and lower estimates for $s$ in the slow regime, based on 
{\bf Lemma \ref{nowaestlem}}, with the parameter $\delta^*$ defined in \eqref{deltastardef}. We also 
utilize the explicit solution of the Michaelis--Menten equation via the Lambert W function, as obtained 
in Schnell \& Mendoza~\cite{SchMen}.

\begin{proposition}\label{penultimatepropeleven}
Denote by $s(t)$ the first component of the solution of \eqref{eqmmirrev} with initial value $(s_0,0)$ 
at $t=0$. Moreover let $\widetilde t\geq t_{\rm cross}$, $\widetilde s:=s(\widetilde t)$, and 
$1>\delta\geq \delta^*$.
\begin{enumerate}[(a)]
\item Define $\underline{s}$, resp.\ $\overline{s}$ by
\begin{equation}\label{mmnwest}
\begin{array}{rcccl}
    \dot{\underline{s}} &= & -\dfrac{k_2e_0\underline{s}}{K_M+\underline{s}}
                                    &,& \underline{s}(\widetilde t)=\widetilde s;\\
    \dot{\overline{s}}  & =&-(1-\delta)\dfrac{k_1k_2e_0\overline{s}}{K_M+\overline{s}}
                                    &,& \overline{s}(\widetilde t)=\widetilde s.\\
\end{array}
\end{equation}
Then, for all $t\geq \widetilde t$, we have 
\begin{equation}\label{sandwich2}
    \underline{s}(t)\leq s(t)\leq\overline{s}(t).
\end{equation}
\item Explicitly, setting
\begin{equation*}
    A:=\dfrac{\widetilde s}{K_M}\exp\dfrac{\widetilde s}{K_M},\quad T:=\dfrac{k_2e_0(t-\widetilde t)}{K_M}
\end{equation*}
we obtain
\begin{equation}\label{Lambertestone}
    \begin{array}{rcl}
    \underline{s}(t)&=& K_M\,W(A\exp(-T))\\
    \overline{s}(t)&=&  K_M\,W(A\exp(-T)\exp(\delta T))\\
    \end{array}
\end{equation}
\end{enumerate}
\end{proposition}
We turn to estimating $\overline{s}-\underline{s}$, using basic properties of the Lambert $W$ function that 
can for instance be found in Mez\H{o}\ \cite[Section~1]{Mez}.

\begin{lemma}\label{slowlamestlem}
With the notation introduced in {\bf Lemma~\ref{penultimatepropeleven}}, one has
\begin{equation}\label{lwestimates}
    \begin{array}{rclcl}
    \overline{s}-\underline{s} &\leq& K_M\log\left(1+W(Ae^{-T})(e^{\delta T}-1)\right)
                & \leq & K_M W(Ae^{-T})(e^{\delta T}-1)\\
    \overline{s}-\underline{s} &\leq& K_M\log\left(1+Ae^{-T}(e^{\delta T}-1)\right)
                & \leq & K_M Ae^{-T}(e^{\delta T}-1)\\
    \end{array}
\end{equation}
for all $t\geq \widetilde t$.
\end{lemma}

\begin{proof}
Let us abbreviate $\alpha:=Ae^{-T}$ and $\beta:=\alpha e^{\delta T}$. Then, with a known identity for $W'$ and 
monotonicity of $W$, one sees
\begin{equation*}
    \begin{array}{rcl}
    K_M^{-1}\left(\overline{s}-\underline{s}\right) 
        &=& \displaystyle \int_\alpha^\beta W'(x)\,{\rm d}x=\int_\alpha^\beta\dfrac{{\rm d}x}{x+\exp(W(x))}\\
    &\leq& \displaystyle \int_\alpha^\beta\dfrac{{\rm d}x}{x+\exp(W(\alpha))} = \log(x+\exp(W(\alpha))|_\alpha^\beta\\
    &=& \log\left(\dfrac{1+\beta e^{-W(\alpha)}}{1+\alpha e^{-W(\alpha)}}\right)
        =\log\left(1+\dfrac{\alpha e^{-W(\alpha)}(e^{\delta T}-1)}{1+\alpha e^{-W(\alpha)}}\right)\\
    &\leq& \log\left(1+\alpha e^{-W(\alpha)} (e^{\delta T}-1)\right)\\
    &=& \log\left(1+W(\alpha) (e^{\delta T}-1)\right),\\
    \end{array}
\end{equation*}
where we have used the defining identity for $W$ in the last step. This shows the first inequality, and 
the remaining ones follow from $0\leq W(x)\leq x$ and $\log(1+x)\leq x$ when $x\geq 0$.
\end{proof}

Presently, we will use only the last inequality from \eqref{lwestimates} to obtain a global error estimate.
\begin{proposition}\label{penultimateproptwelve}
With the assumptions and notation from Proposition \ref{penultimatepropeleven}, for all $t\geq \widetilde t$ 
the following inequalities hold:
\begin{equation}\label{epslwdef}
0 \leq s-\underline{s}\leq \overline{s}-\underline{s}
    \leq s_0\,\exp\left(\dfrac{s_0}{K_M}-1\right)\cdot\dfrac{\delta}{1-\delta}
    \leq s_0\,\exp\left(\dfrac{s_0}{K_M}-1\right)\cdot\dfrac{\delta^*}{1-\delta^*}=:s_0\cdot\varepsilon_{W}.
\end{equation}
\end{proposition}

\begin{proof}
By elementary arguments, the function $T\mapsto e^{-T}(e^{\delta T}-1)$, with derivative 
$T\mapsto e^{-T}\left(1-(1-\delta)e^{\delta T}\right)$ attains its maximum at $T^*=-\log(1-\delta)/\delta\geq 1$, 
with value 
\begin{equation*}
\widetilde s\exp(\widetilde s/K_M)\exp(-T^*)\cdot\dfrac{\delta}{1-\delta}
        \leq s_0\exp(s_0/K_M)\exp(-1)\cdot\dfrac{\delta}{1-\delta}.
\end{equation*}
The assertion follows.
\end{proof}

\begin{remark}{\em  The index in $\varepsilon_{W}$ should remind of its derivation via the Lambert $W$ function. 
This may not be a particularly user-friendly parameter, but one can replace it by more convenient estimates. 
For instance, in case $\varepsilon_{RS}\leq 0.1$, by \eqref{deltastarest} one may choose 
$\delta\leq \frac{10}{9}\varepsilon_{RS}$, and proceed to obtain the estimate
\begin{equation*}
\varepsilon_{LW}\leq \frac{5}{4}\exp\left(\dfrac{s_0}{K_M}-1\right)\cdot\varepsilon_{RS}.
\end{equation*}
}
\end{remark}

\begin{remark}{\em 
For all $t\geq \widetilde t$, we thus obtained the estimates $|s-\underline s|\leq s_0\varepsilon_{W}$, 
and $|s-\underline s|\leq s_0\varepsilon_{L}$. Either of these may be better, given the circumstances. 
\textcolor{black}{Both estimates are rigorous, and moreover $\widetilde t\geq t_{\rm cross}$, for which 
rigorous lower estimates are available.} However, we have to note that their derivation involves some 
simplified estimates, so they may not be optimal. Indeed, extensive numerical experiments point to 
an upper estimate 
\begin{equation}\label{optimalguess}
    \varepsilon_{opt}:=\dfrac{K_S+s_0}{K_M+s_0}\varepsilon_{SSl}\leq\varepsilon_{SSl},
\end{equation}
 but with our toolbox a rigorous proof for this conjecture does not seem possible (see, 
 {{\sc Figure}}~\ref{FIGWW} \textcolor{black}{and also {{\sc Figure}}~\ref{FIGXXZ} below}).}
\end{remark}

\begin{figure}[htb!]
\centering
\includegraphics[width=10.0cm]{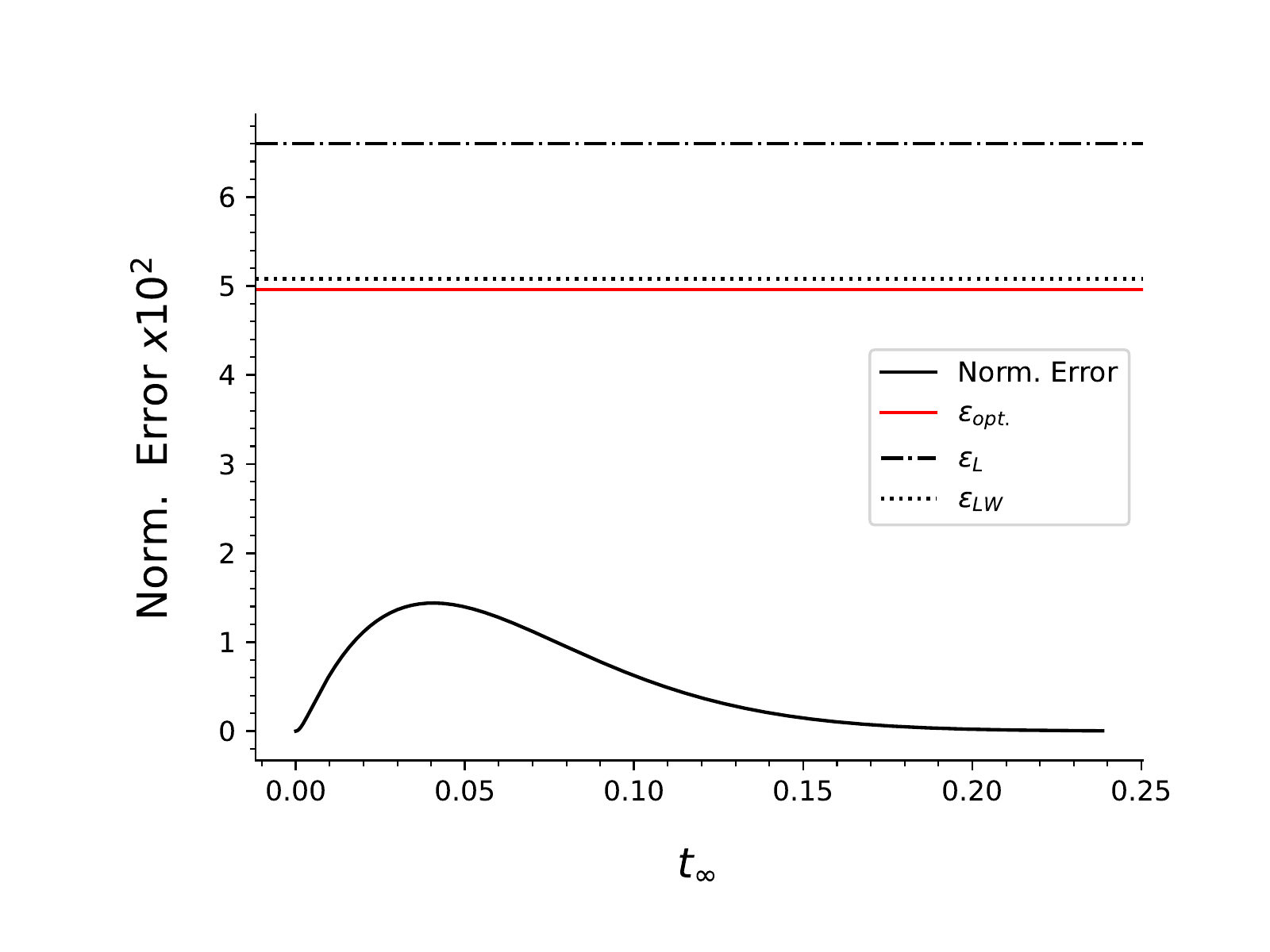}\\
\includegraphics[width=10.0cm]{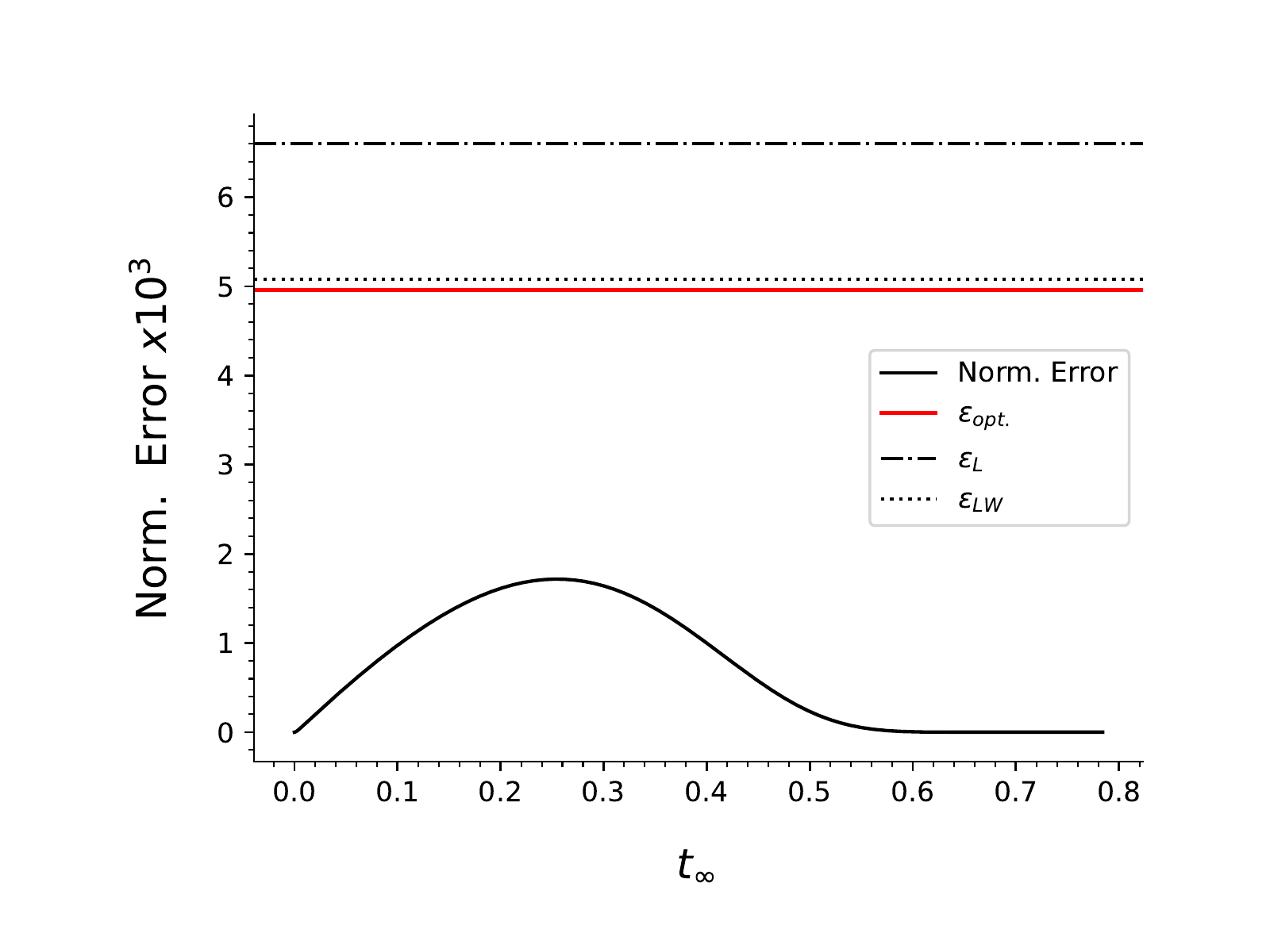}
\caption{\textbf{Numerical simulations suggest that (\ref{optimalguess}) provides an upper bound on the 
normalized error between the $s$-component of the mass action equations and the sQSSA for the complete 
time course when initial conditions lie on the QSS variety, $c=g(s)$.} In both panels, the black curve 
is the numerically-estimated normalized absolute error, $|\xi-s|/s_0$. \textcolor{black}{The dash-dotted 
and dotted lines correspond to $\varepsilon_L$ and $\varepsilon_{LW}$, respectively, and the} red line is 
$\varepsilon_{\rm opt.}$  On the $x$-axis, $t$ has been mapped to $t_{\infty}=1-1/\log(t+e)$, and initial 
conditions for the mass action equations and the sQSSA satisfy $(s,c)(0)=(s,c)(t_{\rm cross})$ and 
$\xi(0)=s(t_{\rm cross})$, respectively ($t_{\rm cross}$ is estimated numerically). {{\sc Top}}: The 
parameters used in the simulation are (in arbitrary units):  $s_0=10.0$, $e_0=10.0$, $k_1=2.0$, 
$k_2=100.0$ and $k_{-1}=100.0$. {{\sc Bottom}}: The parameters used in the simulation are 
(in arbitrary units): $s_0=$\textcolor{black}{$10.0$}, $e_0=1.0$, $k_1=2.0$, $k_2=100.0$ and $k_{-1}=100.0$. 
\textcolor{black}{The estimate $\varepsilon_{\rm opt}$ is not a sharp error estimate due to the choice 
of initial conditions.}}\label{FIGWW}
\end{figure}

\subsubsection{Estimates for the slow dynamics: General case}
Under the hypothesis that $\widetilde t$ and $\widetilde s=s(\widetilde t)$ are known exactly, we obtained upper 
estimates for the approximation error. However, this idealizing assumption does not reflect the real-life 
setting of parameter identification for the reactant stationary approximation. Due to lack of complete information, experimental scientists 
effectively apply the Michaelis--Menten equation with some estimate $s^*$ for $s(\widetilde t)$ valid under
the reactant stationary approximation conditions~\cite{rsa}. This discrepancy must be accounted for by an additional term in the error 
estimate. Define $\xi$ by 
\begin{equation}\label{trueMMeqvar}
\dot \xi=-\dfrac{k_1k_2e_0\xi}{k_{-1}+k_2+k_1\xi}=-\dfrac{k_2e_0\xi}{K_M+\xi},\quad 
    \xi(\widetilde t)=s^*\in (0,\,s_0].
\end{equation}
\begin{proposition}\label{unhappypropvar}
Denote by $s(t)$ the first component of the solution of \eqref{eqmmirrev} with initial value $(s_0,0)$ 
at $t=0$. Moreover, let $\widetilde t\geq t_{\rm cross}$, $\widetilde s:=s(\widetilde t)$. Then, with 
$\underline s$ from \eqref{mmbothestapp} [or from \eqref{mmnwest}], for all $t\geq \widetilde t$, we 
have 
\begin{equation}
    | \xi-\underline{s}|\leq |s^*-\widetilde s|
\end{equation}
and 
\begin{equation}
    |\xi-s|\leq  |s^*-\widetilde s|+ s_0\cdot\varepsilon_{L},
\end{equation}
as well as
\begin{equation}\label{toterreq}
     |\xi-s|\leq  |s^*-\widetilde s|+ s_0\cdot\varepsilon_{W}.
\end{equation}
\end{proposition}
\begin{proof}
For the first inequality use {\bf Lemma~\ref{slowestlem}(c)}. For the second, note that
\[
   |\xi-s|   \leq |\xi-\underline{s}|+|\underline{s}-s|
      \leq|s^*-\widetilde s|+ |\underline{s}-s|
\]
and use {\bf Propositions~\ref{penultimatepropone}} and {\bf\ref{penultimateproptwelve}}, respectively.
\end{proof}

\begin{remark}{\em 
We should make the following observations for the above proposition:
\begin{enumerate}[(a)]
\item {\bf Lemma \ref{slowestlem}} includes an exponentially decaying factor for the first term 
in the estimate. For practical experimental applications, this might be of little relevance for some
enzyme catalyzed reactions, since in the scenario under consideration here this exponential decay 
will be slow and the initial transient will be fast.
\item The special case of \eqref{trueMMeqvar} with $s^*=s_0$ seems to reflect the implicit 
assumption underlying many experiments, i.e., that there is no discernible loss in the transitory 
phase before the starting time $\widetilde t$ for measurements.
\end{enumerate}
}
\end{remark}

With the obvious (and to some extent controllable) choice $\widetilde t=t_u^\dagger(q)$, we obtain 
with {\bf Proposition \ref{stuplem}}:
\begin{corollary}\label{fullerrcor}
Let $0<q<1$ and let $\varepsilon_{SSl}$ satisfy the hypotheses of {\bf Proposition \ref{uppertimeprop}}. 
Then, for all $t\geq t_u^\dagger(q)$, one has
\begin{equation}
\begin{array}{rcl}
\dfrac{|\xi-s|}{s_0}&\leq& \varepsilon_{SSl}\cdot\dfrac{1}{q}
    \log\left(1+\cfrac{1}{q}\dfrac{k_1(K_M+s_0)^2}{k_2K_M}\dfrac{1}{\varepsilon_{SSl}}\right) 
        + \varepsilon_{L};\\
\dfrac{|\xi-s|}{s_0}&\leq& \varepsilon_{SSl}\cdot\dfrac{1}{q}
    \log\left(1+\cfrac{1}{q}\dfrac{k_1(K_M+s_0)^2}{k_2K_M}\dfrac{1}{\varepsilon_{SSl}}\right) 
        + \varepsilon_{W}.\\
\end{array}
\end{equation}
\end{corollary}

\subsubsection{Assuming the standard quasi-steady-state approximation starts at $t=0$}\label{333}
In experiments, it is generally assumed that the substrate concentration does not change during
the initial fast transient. Here we consider a different scenario.  We assume that sQSSA is
applicable from $t=0$. This reflects a widely used scenario in the literature, where one considers the 
reduced Michaelis--Menten equation with initial value $s_0$ at $t=0$ (see, the usual choice 
of initial value for \eqref{classmmeq} in the literature), and compares its solution to the 
true solution. This choice is compatible with the perspective of singular perturbation 
theory, because the relevant solution of \eqref{eqmmirrev} starts on the critical manifold 
with $c=0$. Experimentally it is not an unreasonable approximation, particularly for 
fast-acting enzymes, like carbonic anhydrase. 

We will show for this scenario the approximation error is bounded by a term of order 
$\varepsilon_{SSl}$. More precisely:

\begin{proposition}\label{epsorderprop}
Let $z(t)$ satisfy
\begin{equation*}
\dot{z} = -\cfrac{k_2e_0z}{K_M+z}, \quad z(0)=s_0,
\end{equation*}
and denote by $s(t)$ the first component of the solution of \eqref{eqmmirrev} with initial value $(s_0,0)$ at $t=0$. 
\begin{enumerate}[(a)]
\item Then, for all $t$ with $0\leq t\leq t_{\rm cross}$, one has $z(t)\geq s(t)$ and 
\begin{equation}\label{normERfin}
\cfrac{z-s}{s_0} \leq \varepsilon_{SSl}\cdot 
    \bigg(\cfrac{s_0+K_S}{s_{\rm cross}+K_M}\bigg)\exp \left(k_1s_0\varepsilon_{SSl}t_{\rm cross}\right).
\end{equation}
\item Let $0<q<1$ and $\varepsilon_{SSl}$ satisfy the hypotheses of {\bf Proposition~\ref{uppertimeprop}}.
Then, for all $t$ with $0\leq t\leq t_{\rm cross}$, one has
\begin{equation}\label{normERfinup}
\cfrac{z-s}{s_0} \leq \varepsilon_{SSl}\cdot 
    \cfrac{1}{q}\bigg(\cfrac{s_0+K_S}{s_0+K_M}\bigg)\exp \left(\frac1q k_1s_0t_{SSl}
        \cdot\varepsilon_{SSl}\log\left(1+\dfrac{1}{\varepsilon_{SSl}}\dfrac{C^*}{q}\right)\right),
\end{equation}
with $C^*$ from equation \eqref{tupsteppoint5}.
\end{enumerate}
\end{proposition}
\begin{proof}
Let 
\begin{equation*}
    f(s,c)=: -k_1(e_0-c)s+k_{-1}c,
\end{equation*}
and recall that 
\begin{equation*}
f(s,\,g(s))=-\dfrac{k_2e_0s}{K_M+s},
\end{equation*}
with  $g(s)=e_0s/(K_M+s)$ defined in \eqref{classsmeq}.
As in the proof of {\bf Lemma~\ref{slowestlem}}, one finds
\begin{subequations}
\begin{align*}
    g(z)-g(s) &= e_0\bigg(\cfrac{z}{K_M+s}-\cfrac{s}{K_M+z}\bigg)\\
    &= e_0K_M\cdot \cfrac{z-s}{(K_M+s)(K_M+z)}.
    \end{align*}
\end{subequations}
Now, limit the temporal domain to $t\in [0,t_{\rm cross}]$, which implies $L:=c-g(s)\leq 0$ 
by {\bf Lemma~\ref{schahelem}}, and furthermore
\begin{subequations}
    \begin{align*}
        \dot{s} =f(s,c)&= -k_1e_0s + (k_{-1}+k_1s)c,\\
        & \leq -k_1e_0s + (k_{-1}+k_1s)g(s),\\
        &= -\cfrac{k_2e_0s}{K_M+s}.
    \end{align*}
\end{subequations}
Therefore, $ z \geq s$ for $t\leq t_{\rm cross}$ by the usual differential inequality argument.

With $c=g(s)+L$ one now has
\begin{subequations}
    \begin{align*}
       \cfrac{d}{dt}(z-s) &= f(z,g(z))-f(s,g(s)+L)\\
       &= f(z,g(z))-f(s,g(z)) + f(s,g(z))-f(s,g(s)+L)\\
       &=-k_1(e_0-c)(z-s) + k_1(K_S+s)(g(z)-g(s)-L)\\
       & \leq -k_1(e_0-c)(z-s) + k_1(K_S+z)(g(z)-g(s)-L)\\
       &= -k_1(e_0-c)(z-s) - k_1(K_S+z)L + k_1(K_S+z)\cdot 
            \cfrac{e_0K_M\cdot(z-s)}{(K_M+s)(K_M+z)}\\
       &\leq -k_1(e_0-c)(z-s) -k_1(K_S+z)L +k_1e_0(z-s),\\
    \end{align*}
\end{subequations}
which leaves us with:
\begin{equation}
  \cfrac{d}{dt}(z-s)   \leq k_1c(z-s)-k_1(K_S+s_0)L, \quad \text{for all } t \leq t_{\rm cross}.
\end{equation}
Since $c \leq s_0\varepsilon_{SSl}$ by \eqref{ctildemmlow}, this ultimately implies
\begin{equation}
  \cfrac{d}{dt}(z-s) \leq \varepsilon_{SSl}k_1s_0(z-s)-k_1(K_S+s_0)L, 
        \quad \text{for all } t \leq t_{\rm cross}.
\end{equation}
Now, for $t\leq t_{\rm cross}$, one has with $f(s,c)\leq 0$ and $g'(s)\geq0$ and \eqref{eqmmirrev2}:
\begin{equation*}
    \begin{array}{rcl}
    \dfrac{dL}{dt}     &=& k_1(K_M+s)L-g'(s)\cdot f(s,c) \\
         & \geq& k_1(K_M+s)L\geq k_1(K_M+s_{\rm cross})L,
    \end{array}
\end{equation*}
and therefore
\begin{equation*}
    L\geq -\dfrac{e_0s_0}{K_M+s_0}\,\exp\left(k_1(K_M+s_{\rm cross})t\right).
\end{equation*}
Consequently, we obtain:
\begin{equation}\label{ZSineq}
\cfrac{d}{dt}(z-s) \leq \varepsilon_{SSl}k_1s_0\bigg[(z-s)+(K_S+s_0)
        \exp(-k_1(K_M+s_{\rm cross})t)\bigg], \quad \text{for all }t \leq t_{\rm cross}.
\end{equation}
Solving the corresponding linear differential equation yields for all $t\leq t_{\rm cross}$:
\begin{equation}\label{normER}
\begin{array}{rcl}
     \cfrac{z-s}{s_0} &\leq &\varepsilon_{SSl}\cdot \bigg(\cfrac{K_S+s_0}{K_M+s_{\rm cross}
        +\varepsilon_{SSl}s_0}\bigg)\bigg(\exp (k_1\varepsilon_{SSl}s_0t)-\exp( -k_1(K_M+s_{\rm cross})t)\bigg)\\
     & \leq &\varepsilon_{SSl}\cdot \cfrac{K_S+s_0}{K_M+s_{\rm cross}}\exp (k_1\varepsilon_{SSl}s_0t)\\
      & \leq &\varepsilon_{SSl}\cdot \cfrac{K_S+s_0}{K_M+s_{\rm cross}}\exp (k_1\varepsilon_{SSl}s_0t_{\rm cross}),
\end{array}
\end{equation}
which finishes the proof of part (a).

For part (b), we use $t_{\rm cross}\leq t_u^\dagger(q)$ [see, \eqref{tupdagger}], as well as 
$s_{\rm cross} \geq qs_0$ due to {\bf Lemma \ref{uppertimelem}}.
\end{proof}

Numerical results confirm that (\ref{normER}) yields a rather sharp bound on the normalized 
error accumulated as the phase-plane trajectory approaches the QSS 
manifold (see, {{\sc Figure}}~\ref{FIGXX}).
\begin{remark} {\em 
The following observations should be made about our results:
\begin{enumerate}[(a)]
\item Keeping only the lowest order term in \eqref{normERfinup}, one has for 
$0\leq t\leq t_{\rm cross}$ that
\begin{equation}\label{estimate}
\cfrac{z-s}{s_0} \sim \varepsilon_{SSl}\cdot
    \frac1q \bigg(\cfrac{K_S+s_0}{K_M+s_0}\bigg)+o(\varepsilon_{SSl})
        =:\frac1q\varepsilon_{opt}+o(\varepsilon_{SSl})
\end{equation}
with 
$\varepsilon_{opt}$ defined in \eqref{optimalguess}.
\item Numerical simulations confirm that (\ref{estimate}) is a reliable estimation of the 
normalized error between $z$ and $s$ when $t\leq t_{\rm cross}$ (see, {{\sc Figure}}~\ref{FIGYY}).
\item With {\bf Propositions~\ref{penultimatepropeleven}}, {\bf \ref{penultimateproptwelve}} 
and {\bf \ref{unhappypropvar}} one sees that $|z-s|/s_0$ is of order $\varepsilon_{SSl}$ over 
the whole time range. Numerical simulations suggest that $\varepsilon_{opt}$ is a global upper 
bound (see, {{\sc Figures}}~\ref{FIGYY} and \ref{FIGXXZ}).
\end{enumerate}
}
\end{remark}
\begin{figure}[htb!]
\centering
\includegraphics[width=9.0cm]{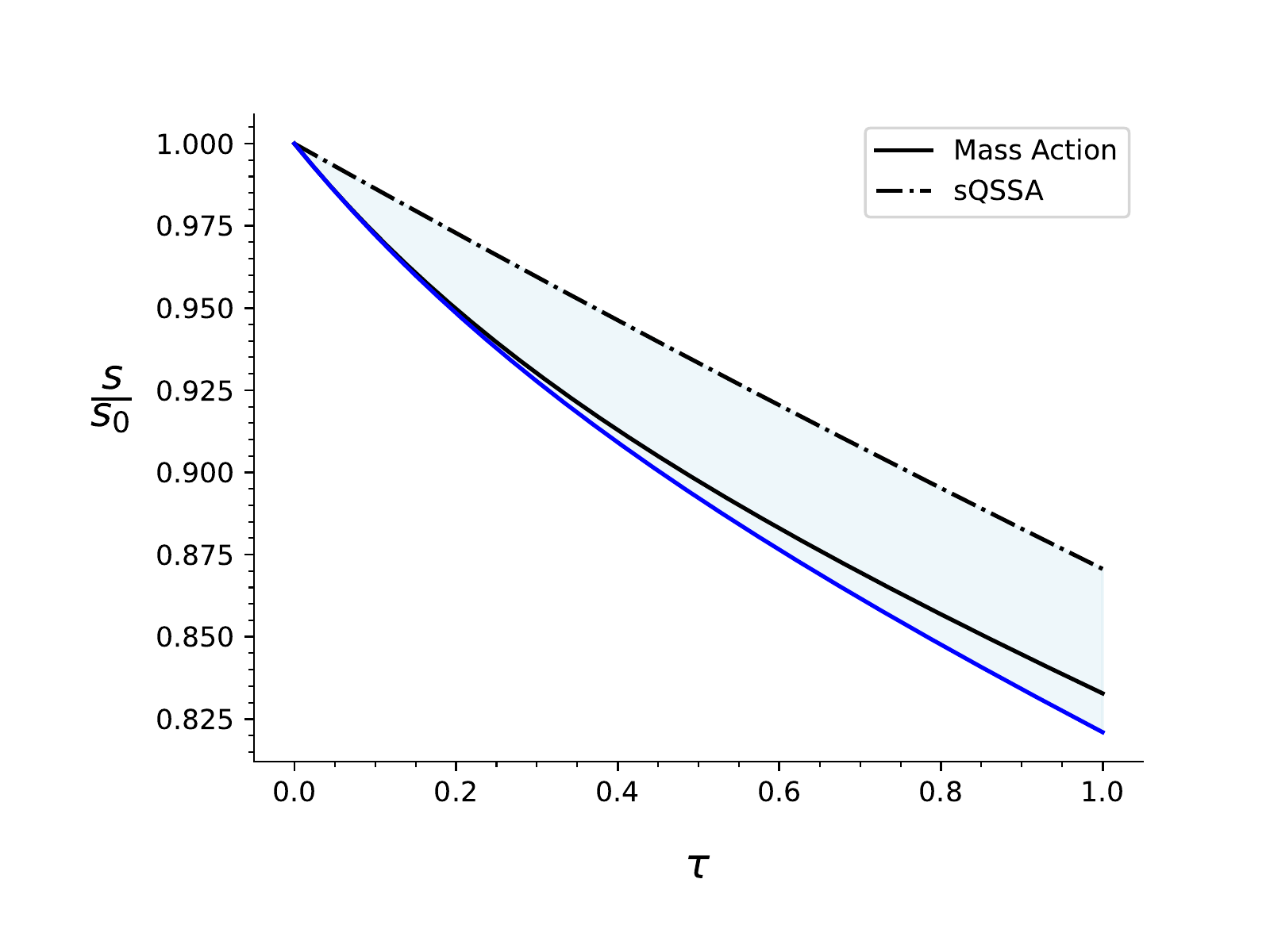}\\
\includegraphics[width=9.0cm]{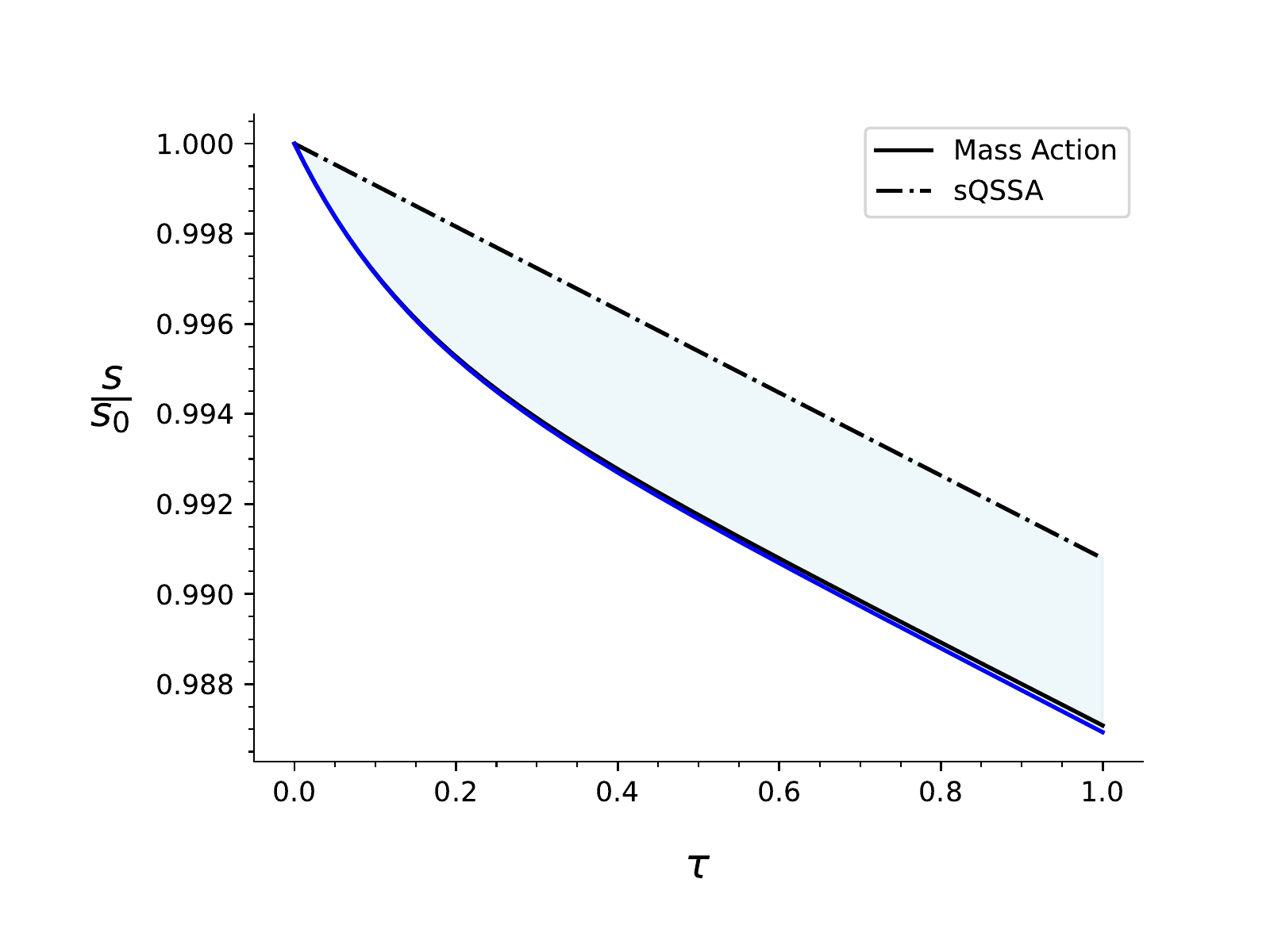}
\caption{\textbf{Numerical simulations confirm that (\ref{normER}) provides a sharp bound on 
the normalized error between the $s$-component of the mass action equations and the sQSSA.} 
In both panels, the black curve is the numerical solution the mass action equations. The 
dashed/dotted curve is the numerical solution to the sQSSA. \textcolor{black}{The admissible 
region given by the error bound (\ref{normER}) is shaded in blue.} The blue line is 
the (normalized) numerical solution to the right-hand side of (\ref{normER}) with 
numerically-estimated (a priori) $q$ that corresponds to the upper boundary of (\ref{normER}). 
Time has been rescaled by $\tau = t/t_{\rm cross}$, where $t_{\rm cross}$ has been 
numerically-estimated. {{\sc Top}}: The parameters used in the simulation are (in arbitrary 
units): $s_0=10.0$, $e_0=10.0$, $k_1=2.0$, $k_2=100.0$ and $k_{-1}=100.0$. {{\sc Bottom}}: The 
parameters used in the simulation are (in arbitrary units): $s_0=100.0$, $e_0=1.0$, $k_1=2.0$,
$k_2=100.0$ and $k_{-1}=100.0$.
 }\label{FIGXX}
\end{figure}
\begin{figure}[htb!]
\centering
\includegraphics[width=8.0cm]{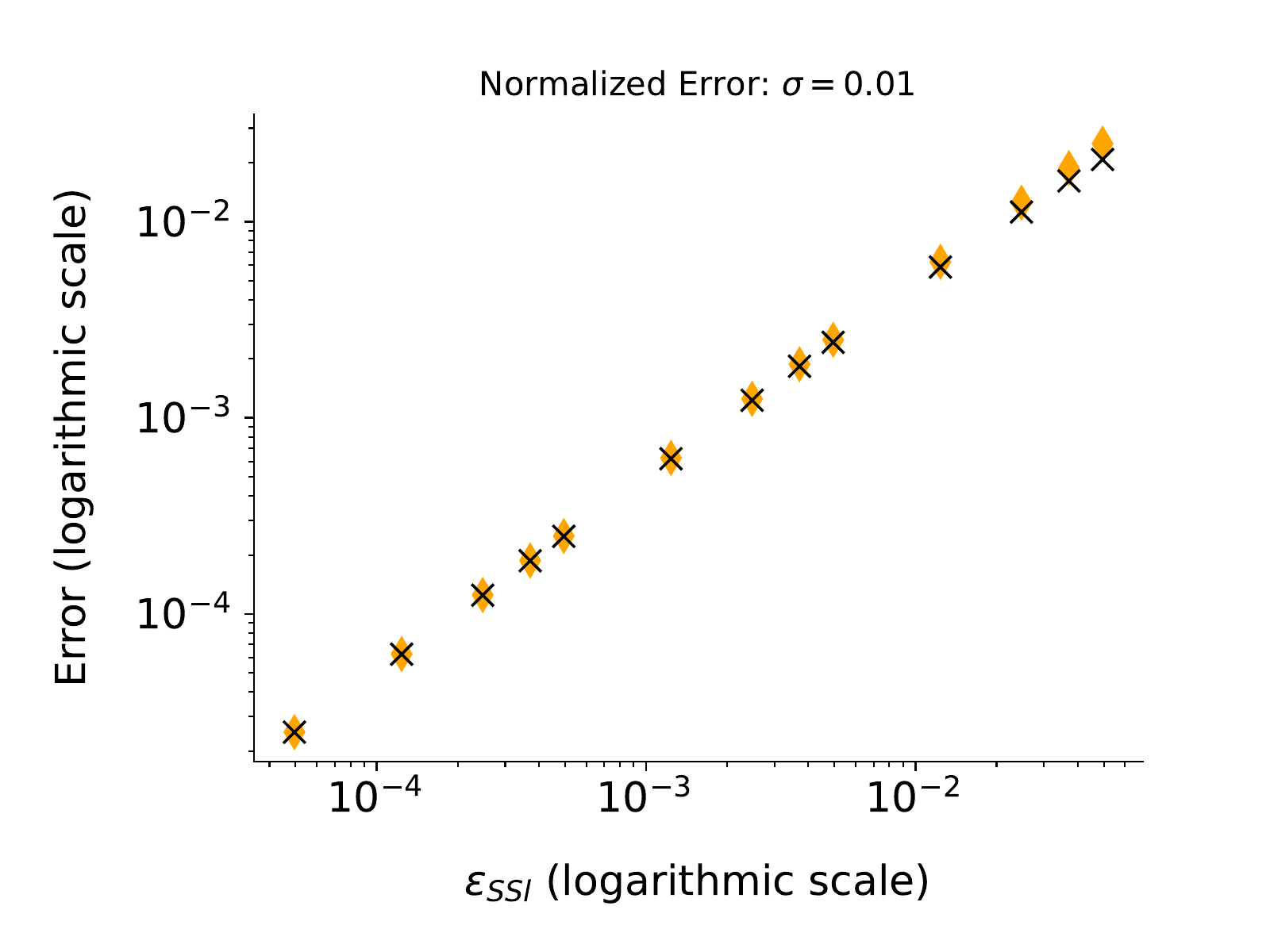}
\includegraphics[width=8.0cm]{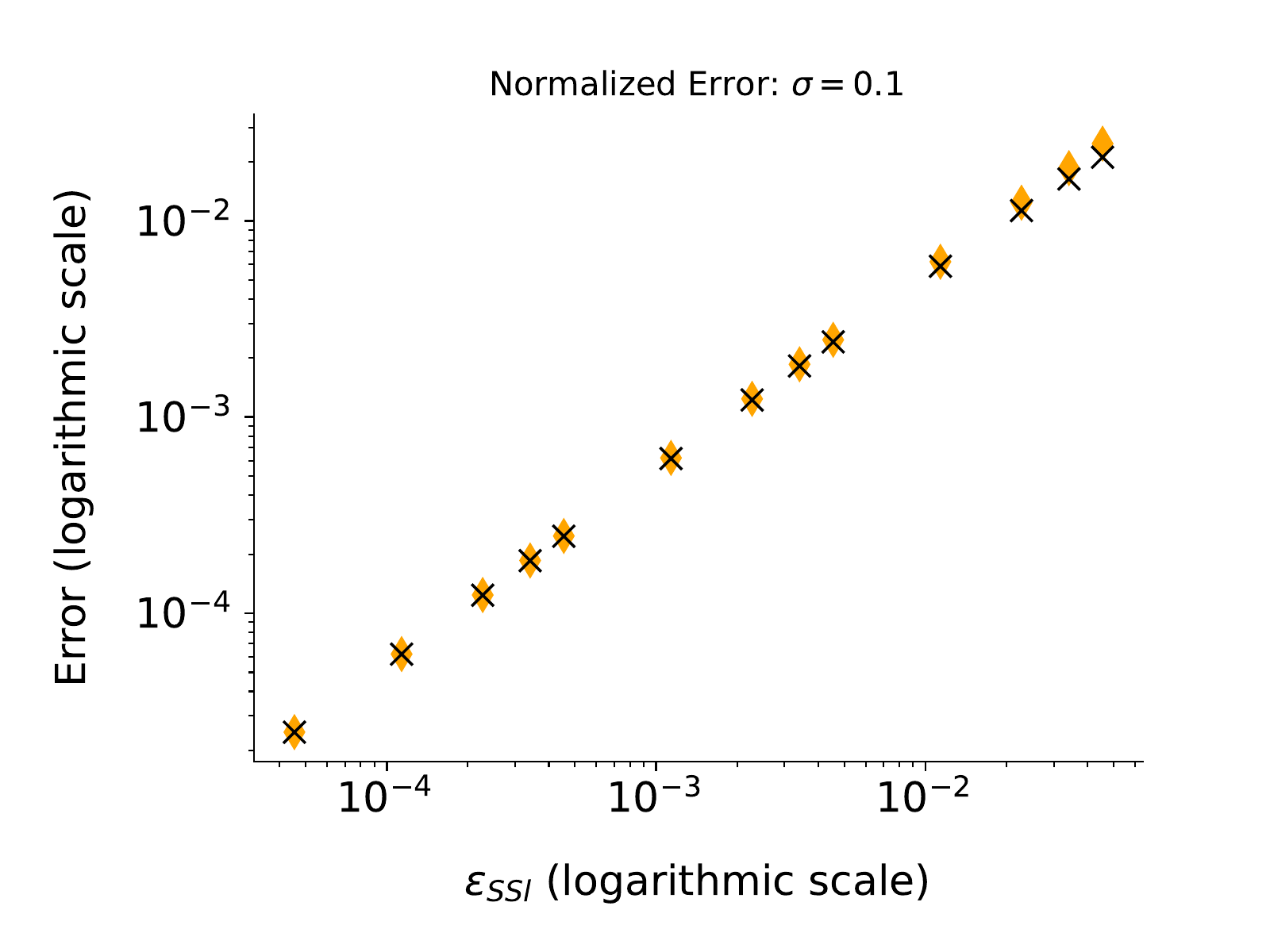}\\
\includegraphics[width=8.0cm]{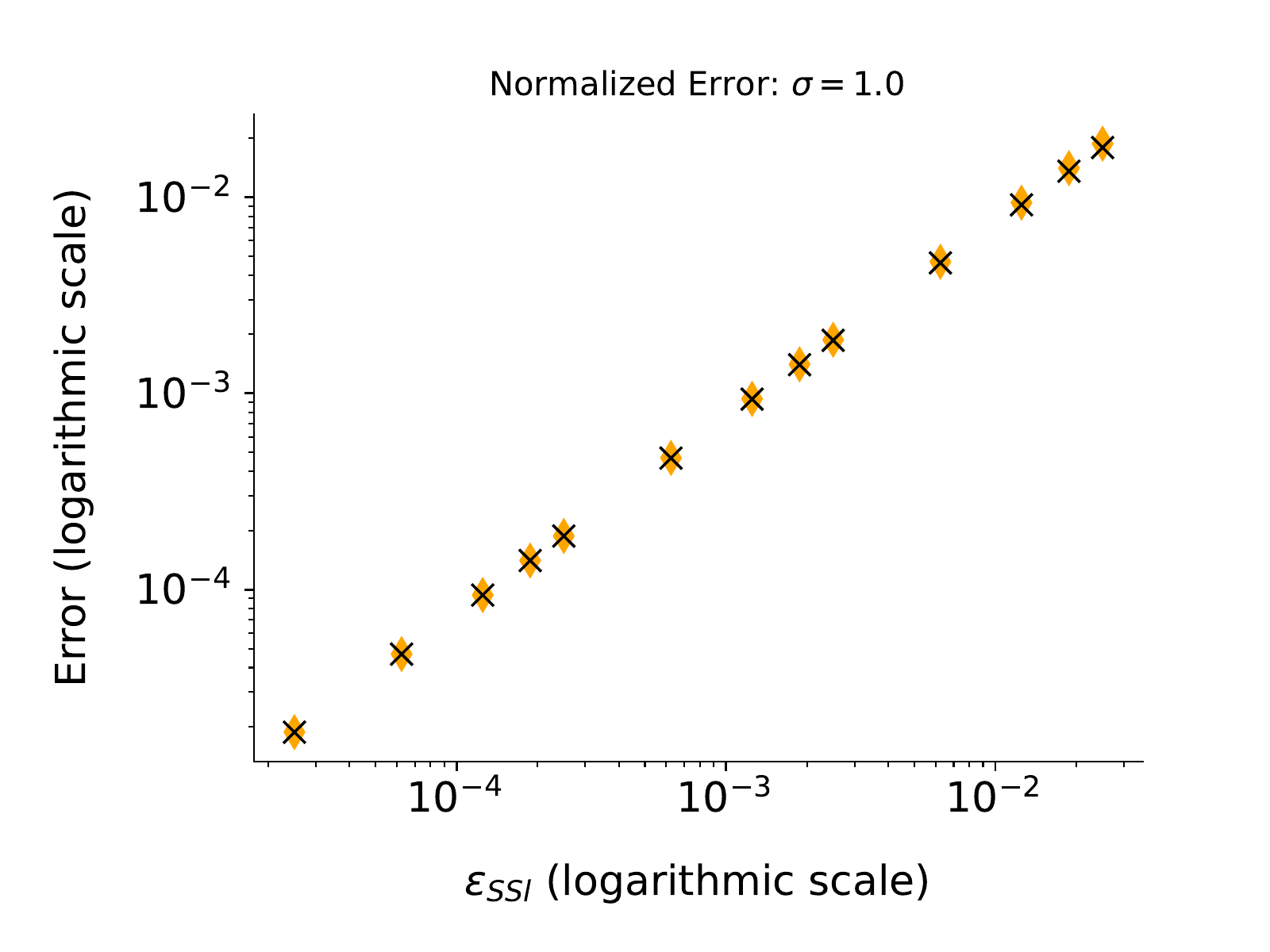}
\includegraphics[width=8.0cm]{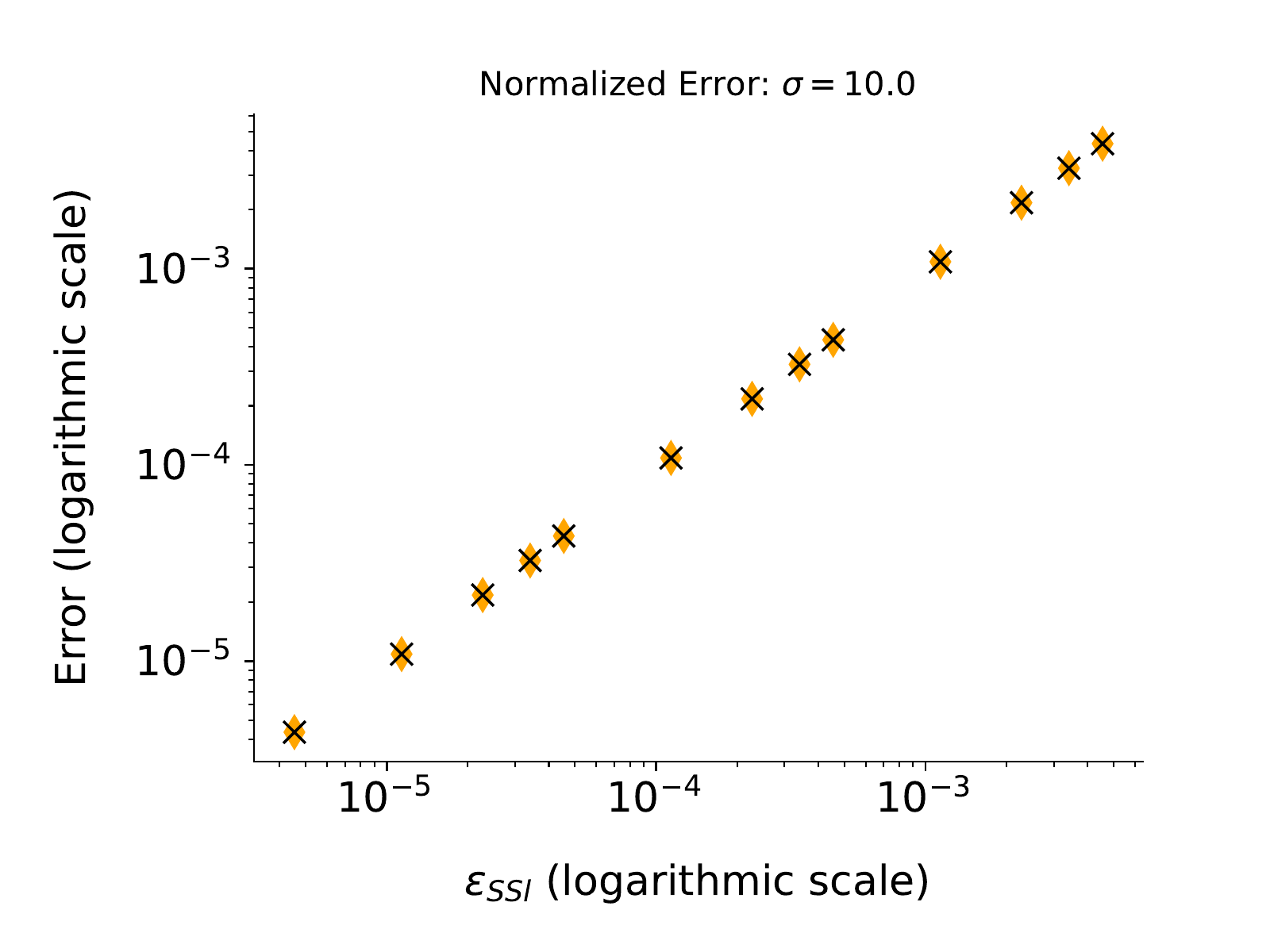}
\caption{\textbf{Numerical simulations confirm that (\ref{optimalguess}) is a reasonable 
estimation of the normalized error when $ t=t_{\rm cross}$ for small $\varepsilon_{SSl}$}. In 
all panels, $e_0 \in [0.025, 0.05, 0.075, 0.1, 0.25, 0.5, 0.75, 1.0, 2.5, 5.0, 7.5, 10]$, 
$k_1=1.0,$  $k_2=k_{-1}=100.0$, thus $K_M=200$, and $\sigma=s_0/K_M$. The solid black 
crosses are the numerically--computed normalized error $|z-s|/s_0$ at $t=t_{\rm cross}$.
The orange diamonds correspond to $\varepsilon_{opt}$. {{\sc Top Left:}} $s_0 =2.0$. 
{{\sc Top Right:}} $s_0=20.0$. {{\sc Bottom Left:}} $s_0=200.0$. {{\sc Bottom Right:}} 
$s_0=2000.0$. 
}\label{FIGYY}
\end{figure}

\begin{remark}{\em 
The distinguishing difference between $\varepsilon_{SSl}$ and (\ref{estimate}) is the 
appearance of the dimensionless factor
\begin{equation*}
  \eta =: \bigg(\cfrac{K_S+s_0}{K_M+s_0}\bigg).
\end{equation*}
Recall that the specificity constant \cite{specificity}, $\Theta$, is defined as
\begin{equation}\label{Ef}
    \Theta =: \cfrac{k_2}{K_M} \leq k_1.
\end{equation}
From (\ref{Ef}) we can define the \textit{normalized} specificity constant, 
$\bar{\Theta}=:\Theta/k_1$. Expressing $\eta$ in terms of $\bar{\Theta}$ and setting 
$\sigma:=s_0/K_M$ yields
\begin{equation}
    \eta = \cfrac{1+\sigma - \bar{\Theta}}{1+ \sigma},
\end{equation}
and we conclude that (\ref{estimate}) will be much smaller than $\varepsilon_{SSl}$ 
whenever $\bar{\Theta} \approx 1+ \sigma$, which implies that $\sigma \ll 1$ and 
$\bar{\Theta}$ is close to $1$. This scenario can be useful in the study of functional 
effects of enzyme mutations \cite{mutant}. Numerical simulations confirm that the 
normalized error may be far less than $\varepsilon_{SSl}$ when $\eta \ll 1$ 
(see, {{\sc Figure}} \ref{FIGZZ}).}
\end{remark}

\begin{figure}[htb!]
\centering
\includegraphics[width=9.0cm]{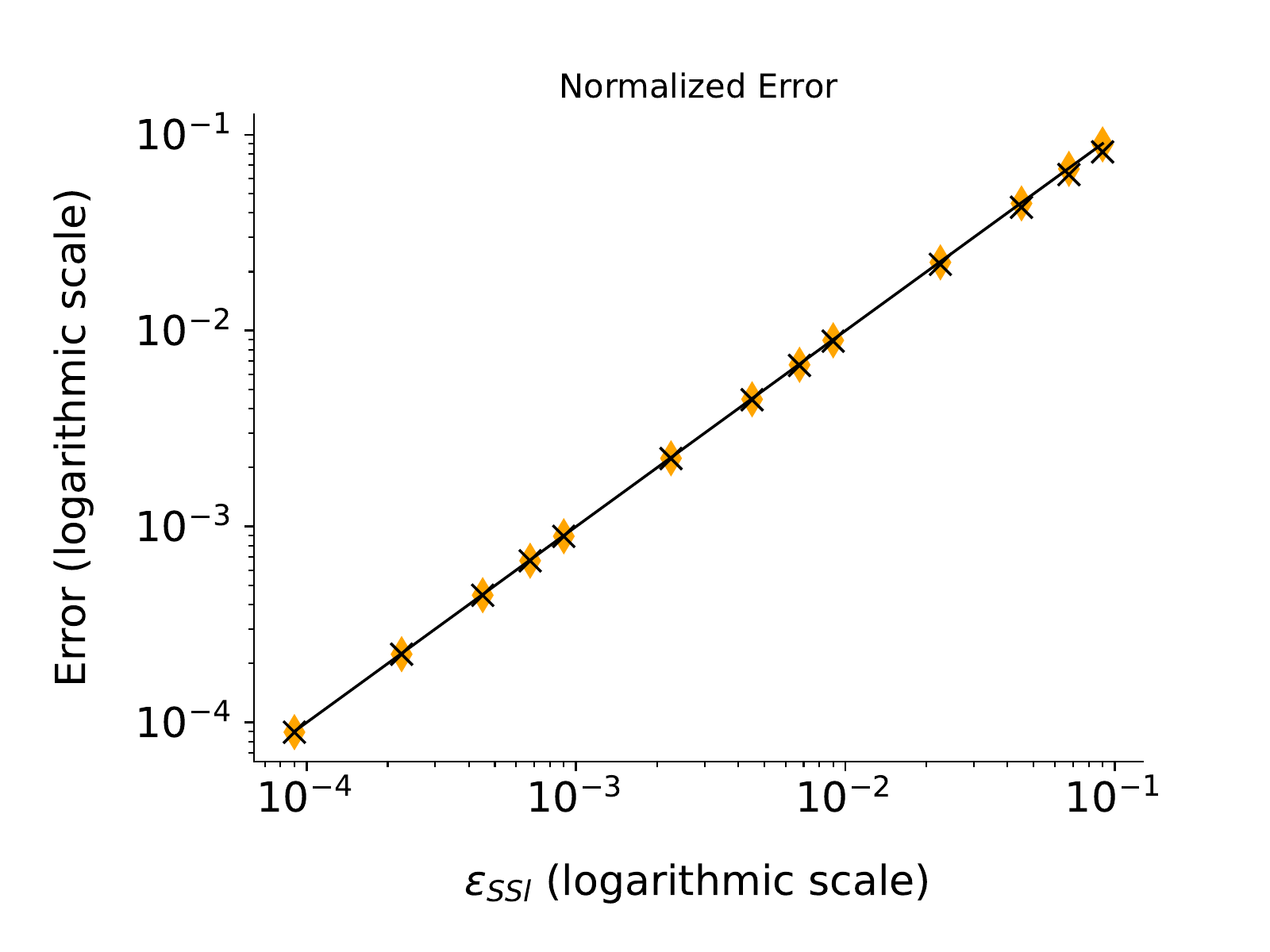}\\
\includegraphics[width=9.0cm]{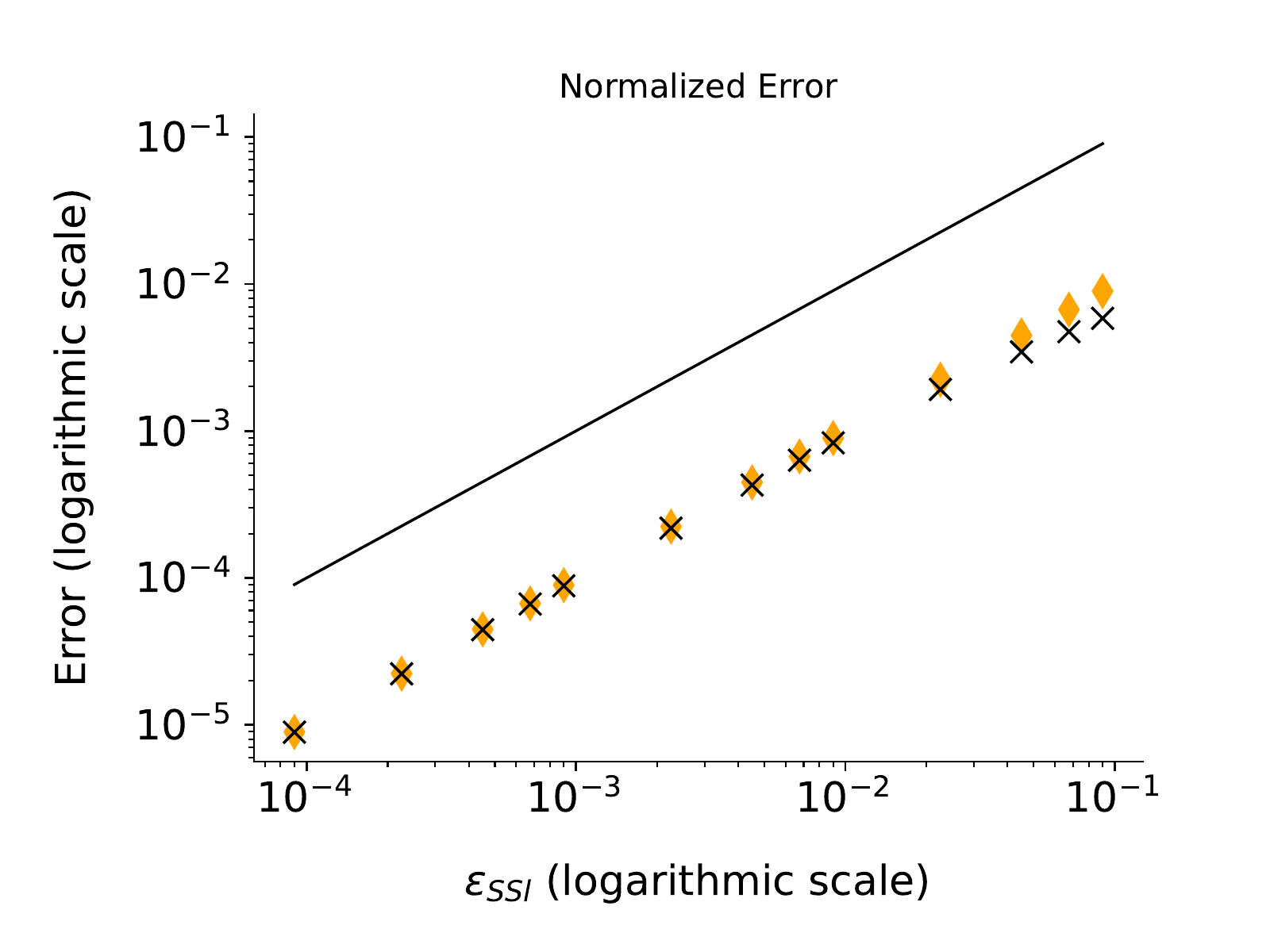}
\caption{\textbf{Numerical simulations confirm that (\ref{optimalguess}) provides a 
sharp bound on the normalized error at $t=t_{\rm cross}$ between the $s$-component of the mass action 
equations and the sQSSA.} In both panels the black line corresponds to 
$\log$(Error) $=\log(\varepsilon_{SSl})$ and 
$e_0 \in [0.025, 0.05, 0.075, 0.1, 0.25, 0.5, 0.75, 1.0, 2.5, 5.0, 7.5, 10]$. The 
orange diamonds correspond to (\ref{estimate}) and the black crosses are the 
numerically-estimated normalized error between $s$ and $z$ at the numerically-estimated 
crossing time. {{\sc Top}}: The parameters used in the simulation are (in arbitrary units): 
$s_0=10.0$, $e_0=1.0$, $k_1=2.0$, $k_2=1.0$ and $k_{-1}=100.0$. {{\sc Bottom}}: The 
parameters used in the simulation are (in arbitrary units): $s_0=10.0$, $e_0=1.0$,
$k_1=2.0$, $k_2=100.0$, and $k_{-1}=1.0$. Note that $\sigma \approx 0.1$ in both 
simulations. In the top panel, $\eta \approx 0.99$ and therefore the normalized error 
is approximately $\varepsilon_{SSl}$. On the other hand, $\eta \approx 0.1$ in the 
bottom panel, and therefore the normalized error at the crossing time is roughly one 
order of magnitude less than $\varepsilon_{SSl}.$}\label{FIGZZ}
\end{figure}

\begin{figure}[htb!]
\centering
\includegraphics[width=10.0cm]{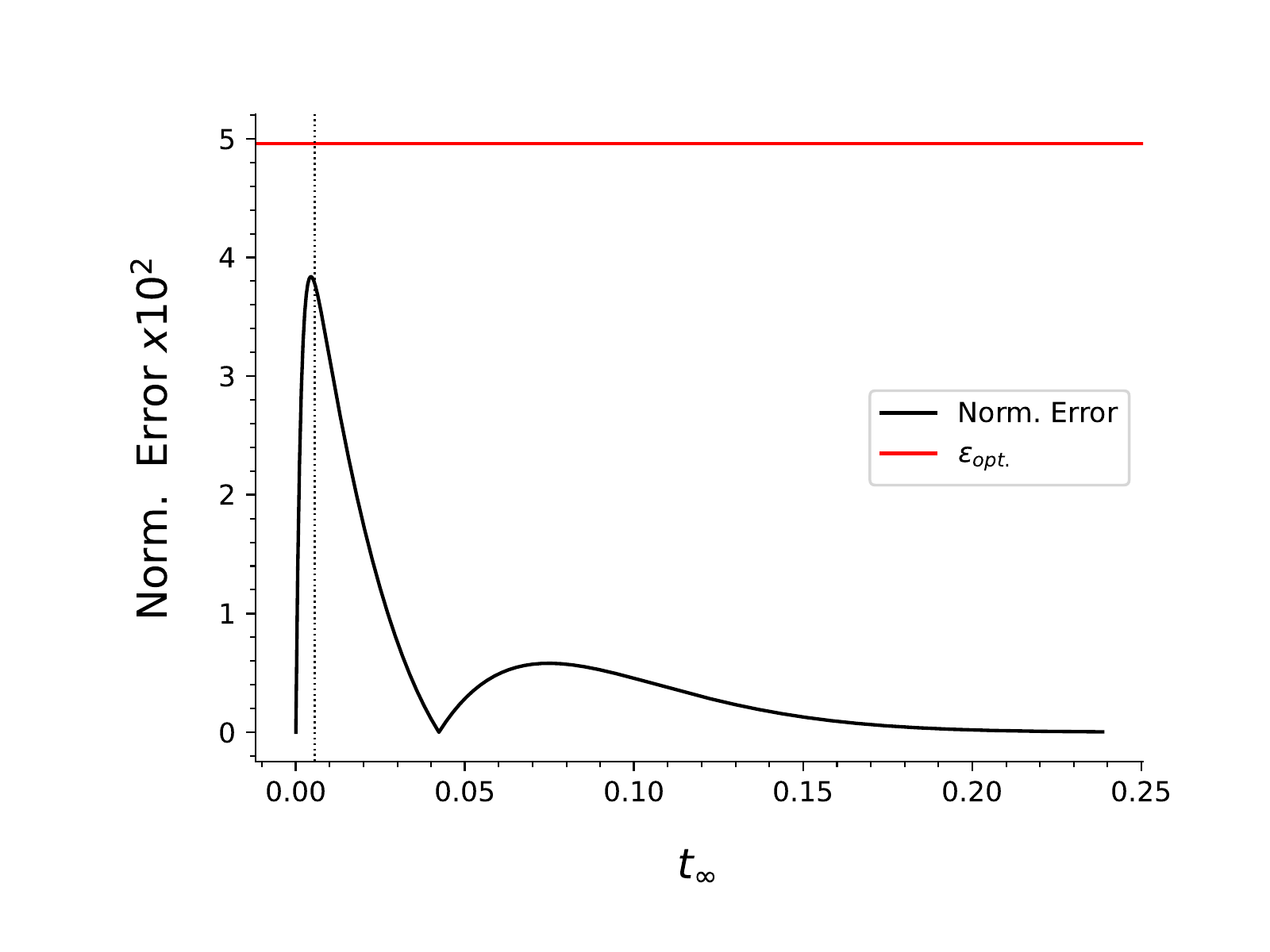}\\
\includegraphics[width=10.0cm]{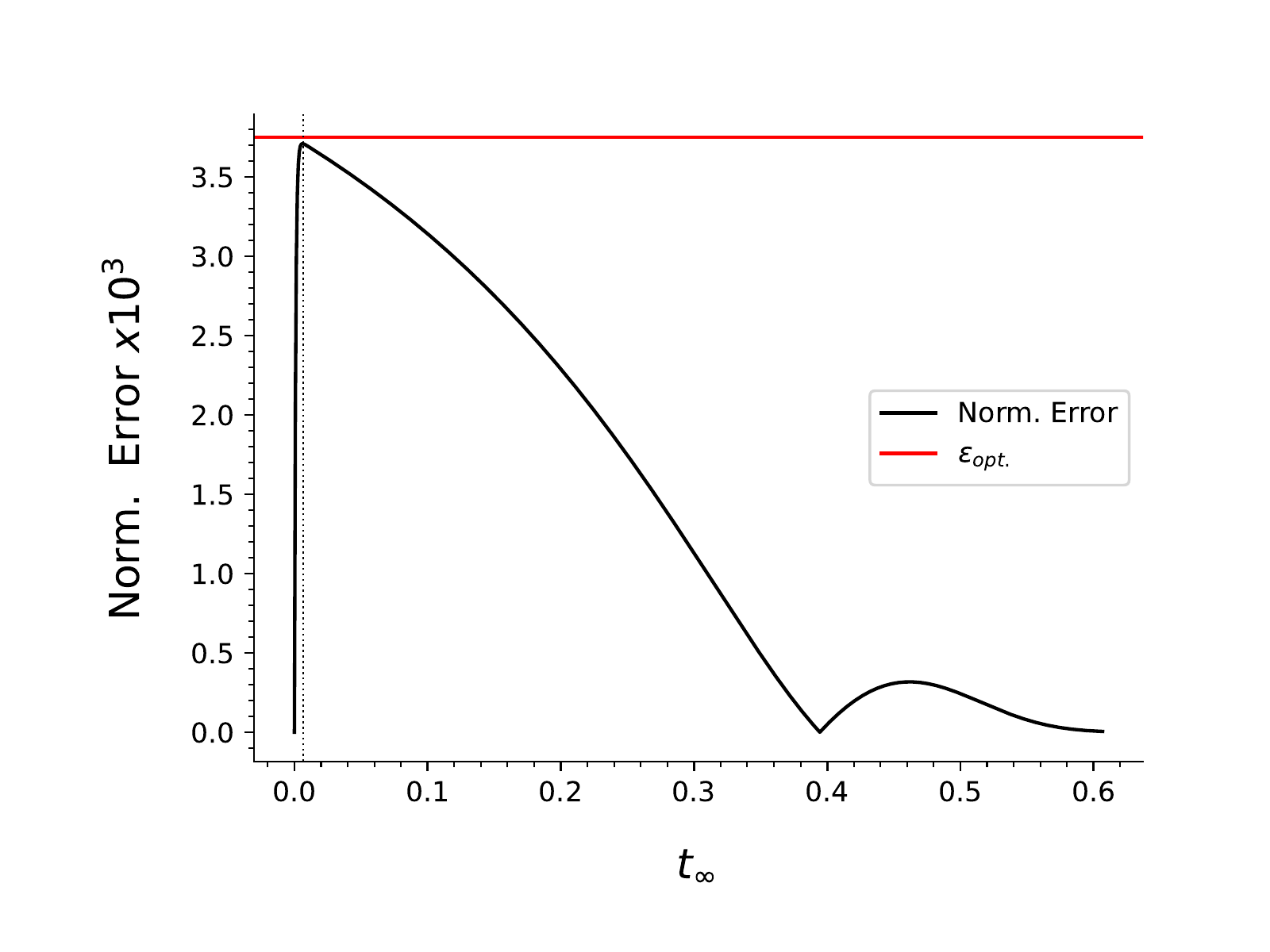}
\caption{\textbf{Assuming reactant stationary approximation from $t=0$: Numerical simulations suggest that (\ref{optimalguess}) 
provides a sharp bound on the normalized error between the $s$-component of the mass action 
equations and the sQSSA for the complete time course.} In both panels, the black curve 
is the normalized absolute error $|s-z|/s_0$. The red line is $\varepsilon_{\rm opt.}$ 
The initial conditions are $(s,c)(0)=(s_0,0)$ and $z(0)=s_0$ and thus correspond to 
the assumption that the reactant stationary approximation is valid at the start of the reaction outlined 
in Section~\ref{333}. The dotted vertical line demarcates the numerically--computed 
$t_{\rm cross}$. Note that $t$ has been mapped to $t_{\infty}=1-1/\log(t+e)$. 
{{\sc Top}}: The parameters used in the simulation are (in arbitrary units): 
$s_0=10.0$, $e_0=10.0$, $k_1=2.0$, $k_2=100.0$ and $k_{-1}=100.0$. 
{{\sc Bottom}}: The parameters used in the simulation are (in arbitrary units): 
$s_0=100.0$, $e_0=1.0$, $k_1=2.0$, $k_2=100.0$ and $k_{-1}=100.0$.
}\label{FIGXXZ}
\end{figure}
\subsection{About the long-time quality of the approximation}
The goal of the present work was to obtain workable upper estimates for the relative approximation error, 
$|(*-s)/s_0|$, where $*$ symbolizes the solution of some reduced equation, that are valid over the whole 
range of the slow dynamics. In their derivation, we deliberately chose simplified estimates which do not 
reflect that substrate concentration approaches $0$ as $t\to\infty$. Notably, in 
{\bf Lemma \ref{slowestlem}} and {\bf Lemma \ref{slowlamestlem}}, we eventually disregarded slowly decaying 
terms which would imply convergence to zero for the approximations. So in these estimates the dynamics is 
not reflected well for very long times. (We recall that the parameters $\varepsilon_{MM}$ and 
$\varepsilon_{RS}$ govern the accuracy of the approximation for very long times; see 
Eilertsen et al.~\cite{Chihuahua}.) Our simplifications are justified since for the intended application -- 
parameter identification -- the time range directly after the onset of the slow dynamics is relevant, 
while the behavior as $t\to\infty$ is of less interest in the experimental setting.

\section{Discussion: A view toward applications} \label{sec:applications}
Experimental enzymologists, and biochemists and analytical chemists may be less interested in 
mathematical technicalities and wish to focus on the essential results. Therefore, we will here 
summarize some essential application-relevant consequences from our theoretical considerations. 
These takeaways will remain technical for experimental scientists, but they are accessible
to mathematical biologists and chemists, who work in close collaboration with experimental 
scientists. We provide quantitative error estimates, which may be relevant 
for a detailed study in application scenarios. In order to present the results without recourse 
to the technical sections, we will accept some redundancy. \textcolor{black}{A quick-reference 
for the parameters defined in this paper can be found in the Appendix.}

\textcolor{black}{Thus, we consider the Michaelis-Menten system \eqref{eqmmirrev} for low initial 
enzyme concentration, so that the quasi-steady state approximation \eqref{classmmeq} holds with 
good accuracy. 
As a distinguished perturbation parameter we choose 
\begin{equation*}
    \varepsilon_{SSl}=\dfrac{e_0}{K_M+s_0},
\end{equation*}
as proposed by Segel and Slemrod \cite{SSl},\footnote{As mentioned in the Introduction, this choice 
is a matter of convenience.} and discuss the limiting case $\varepsilon_{SSl}\to 0$ (in detail, 
$e_0\to 0$ with the other parameters bounded above and below by positive constants).}

\textcolor{black}{
Since $e_0$ and $s_0$ are controllable in a laboratory setting, the standard \textcolor{black}{parameter estimation} will 
provide approximate values for $K_M$ and $k_2$. In turn, this enables an educated guess for 
$\varepsilon_{SSl}$.}

\textcolor{black}{
Our results provide error estimates for the approximation by the Michaelis-Menten 
equation~\eqref{classmmeq}. This may be taken as a vantage point toward error estimates 
for $k_2$ and $K_M$, and thus for consistency checks. Going beyond this, we obtain rigorous 
estimates for the onset time of the QSS regime, and for the substrate depletion during the 
initial transitory phase. This opens an approach to the identification of critical parameters
required for the effective design of steady-state experiments.}

\textcolor{black}{
We will only exhibit the two lowest-order terms in the asymptotic expansions 
with respect to $\varepsilon_{SSl}$, since these are dominant for sufficiently small initial 
enzyme concentration, and then further simplify these terms. The accuracy of 
approximation can --- in any case --- be gauged by a fuller analysis of the results in Section~\ref{estsec}.} 

\subsection{Onset of the slow dynamics}
Generally, via singular perturbation theory one cannot define a fixed time 
for the end of the transitory phase. There always remains some freedom of choice when implementing
a scale. As a definitive (biochemically relevant) time for the onset of the slow dynamics of the 
Michaelis--Menten reaction mechanism, following precedent, we chose the crossing time 
$t_{\rm cross}$, at which complex concentration is maximal. As noted in {\bf Remark~\ref{nossltime}}, 
the familiar time 
\begin{equation*}
    t_{SSl}=\dfrac{1}{k_1(s_0+K_M)}
\end{equation*}
which seems to be suggested by Segel and Slemrod \cite[equation~(12)c]{SSl}, for the duration of 
the transient phase, leads to an underestimate for the asymptotics.  

In this paper, we found:
\begin{itemize}
\item According to \eqref{tlowdag} and \eqref{tlowdagas}, an asymptotic lower estimate for the 
crossing time is given by
\begin{equation*}
t_\ell^*= t_{SSl}\,\left[\log\left(\dfrac{1}{\varepsilon_{SSl}}\right)
                +\log\dfrac{k_{-1}+k_2}{k_2}+\cdots\right]
\end{equation*}
\item According to \eqref{tupdagger} and \eqref{tupdaggeras}, an asymptotic upper estimate for 
the crossing time is given by
\begin{equation*}
t_u^*(q)=  \frac1qt_{SSl}\left[\log\dfrac{1}{\varepsilon_{SSl}}
                +\log\dfrac{k_1(s_0+K_M)^2}{qk_2K_M}+\cdots\right],
\end{equation*}
where $q<1$ is fixed but may be taken arbitrarily close to $1$. To simplify this, we may 
approximate the upper estimate by
\begin{equation*}
t_u^*(1)=  t_{SSl}\left[\log\dfrac{1}{\varepsilon_{SSl}}
                +\log\dfrac{k_1(s_0+K_M)^2}{k_2K_M}+\cdots\right],
\end{equation*}
since by {\bf Lemma~\ref{tuto1}}, we may control the relative error by $1-q$ times a factor 
close to 1.
\end{itemize}
These considerations show that a lowest order approximation of the crossing time, hence of 
the onset time of slow dynamics, is given by
\begin{equation}\label{tcrosssimple}
t_{\rm cross} \approx t^{\ast}:= t_{SSl}\,\log\left(\dfrac{1}{\varepsilon_{SSl}}\right).
\end{equation}

\subsection{Substrate depletion in the transient phase}
Now, we estimate the relative substrate loss at $t_{\rm cross}$, thus 
$\Delta:=\dfrac{s_0-s_{\rm cross}}{s_0}$ to estimate the validity of the \textcolor{black}{reactant 
stationary approximation}. In Section~\ref{estsec}, 
we found:
\begin{itemize}
\item \textcolor{black}{From \eqref{lowscr} and \eqref{lowscras}, one gets an asymptotic lower estimate
\[
\Delta \geq \dfrac{K}{2(s_0+K_M)}\varepsilon_{SSl}\left[\log\dfrac{1}{\varepsilon_{SSl}}
        +\log\dfrac{k_1K_M}{k_2} +\cdots \right],
\]
noting (with {\bf Remark~\ref{badsellrem})} that the factor $\frac12$ could be discarded 
under slightly stricter hypotheses. The factor $K/(s_0+K_M)$ stems from the 
estimate \eqref{notsosimplesest} in {\bf Lemma \ref{simplesestlem}}. As noted in 
{\bf Remark \ref{badsestrem}}, the latter is not optimal.}
\item From \eqref{redinitvalest} and \eqref{redinitvalestas}, one gets an asymptotic 
upper estimate
\[
\Delta\leq  \dfrac1q\varepsilon_{SSl}\left[\log\dfrac{1}{\varepsilon_{SSl}}
        +\log\dfrac{k_1(s_0+K_M)^2}{qk_2K_M}+\cdots \right],
 \]
 with $q<1$ but arbitrarily close to $1$.
\end{itemize}
With similar arguments as those in {\bf Lemma \ref{tuto1}}, one sees that replacing 
$\Delta$ by
\begin{equation}
\dfrac{s_0-s_{\rm cross}}{s_0}\approx \Delta^*:= 
    \varepsilon_{SSl}\left[\log\dfrac{1}{\varepsilon_{SSl}}
        +\log\dfrac{k_1(s_0+K_M)^2}{k_2K_M}+\cdots\right]
\end{equation}
involves a relative error equal to $1-q$ times a factor close to $1$. From the derivations 
via differential inequalities, there remains a gap between $\Delta^*$ and the lower estimate 
$K/(s_0+K_M)\Delta^*$. Here, we resort to a heuristic argument. Given the accuracy of lower 
and upper estimates in {\bf Lemma \ref{simplesestlem}} [note {\bf Remark \ref{badsestrem}}], 
it seems preferable to choose $\Delta^*$ as the appropriate approximation. 
\textcolor{black}{(This choice is also supported by extensive numerical simulations.)} Keeping only the 
lowest order term, we note the approximation
\begin{equation}\label{scrossimple}
  \dfrac{s_0-s_{\rm cross}}{s_0}\approx \Delta^{**}
        :=\varepsilon_{SSl}\log\dfrac{1}{\varepsilon_{SSl}}
\end{equation}
which depends only on the Segel--Slemrod parameter.

\textcolor{black}{
An educated guess for $s_{\rm cross}$ can be obtained based on $\varepsilon_{SSl}$. If progress 
curves are carried out in the laboratory, this opens a way to determine an approximation to the 
crossing time $t_{\rm cross}$, and further on $t_{SSl}$ and the reaction parameter $k_1$ as well 
as $k_{-1}$.}

\subsection{The approximation error assuming no transient substrate loss}
\textcolor{black}{
The approximation of the $s$ component of the solution of \eqref{eqmmirrev} by the solution of 
\eqref{classmmeq} (after the transient phase) is correct only up to some error, and we determined 
rigorous bounds for this error. In turn, this information may be used toward estimating the errors 
in the determination of $K_M$ and $v_\infty$ from experimental data. First, we consider} 
the scenario assuming no loss of substrate in the transient phase: 
$s(0\leq t\leq t_{\rm cross})=s_0$. This is considered the standard reactant 
stationary approximation scenario in
enzyme kinetics. Allowing for somewhat weaker estimates by taking the limiting 
case with $q=1$ and discarding higher order terms as $e_0\to 0$, we arrive at 
``ultimate small parameters'' for estimating the approximation error. The first 
step yields, depending on {\bf Proposition~\ref{penultimatepropone}} or 
{\bf Proposition~\ref{penultimateproptwelve}}, respectively:
\begin{align}
\varepsilon^\dagger_L &:= \varepsilon_{SSl}
    \left(\log\left(\dfrac{k_1(K_M+s_0)^2}{k_2K_M}\dfrac{1}{\varepsilon_{SSl}}\right) 
        +\dfrac{(K_M+s_0)^2(K_S+s_0)}{K_M^3}\right), \\
{\rm or} \nonumber \\
\varepsilon^\dagger_M &:= \varepsilon_{SSl}
    \left(\log\left(\dfrac{k_1(K_M+s_0)^2}{k_2K_M}\dfrac{1}{\varepsilon_{SSl}}\right) 
        +\exp\left(\frac{s_0}{K_M}-1\right)\dfrac{(K_M+s_0)}{K_M}\right).
\end{align}
In a second step, we keep only lowest order terms in the asymptotics of $\varepsilon^\dagger$. With
\begin{equation*}
\log\left(\dfrac{k_1(K_M+s_0)^2}{k_2K_M}\dfrac{1}{\varepsilon_{SSl}}\right)
        =\log\left(\dfrac{k_1(K_M+s_0)^2}{k_2K_M}\right)
                +\log\left(\dfrac{1}{\varepsilon_{SSl}}\right)
\end{equation*}
we ultimately obtain
\begin{equation}\label{estsimple}
    \varepsilon^{\ddagger}=\varepsilon_{SSl}\log\left(\dfrac{1}{\varepsilon_{SSl}}\right).
\end{equation}
Remarkably, in the asymptotic limit the error due to substrate depletion in the transitory 
phase (which is responsible for the logarithmic term) is dominant. 

\subsection{The approximation error assuming standard quasi-steady-state approximation starts at $t=0$}
\textcolor{black}{
In {\bf Proposition \ref{epsorderprop}}, from this assumption we obtained an asymptotic error 
estimate~\eqref{estimate} that is of order $\varepsilon_{SSl}$. }
Combining this with {\bf Proposition~\ref{penultimatepropone}}, resp.\ 
{\bf Proposition~\ref{penultimateproptwelve}}, and keeping only lowest order terms, we obtain 
\begin{align}
\varepsilon^{\S}_L &:= \varepsilon_{SSl}\left(\dfrac{K_S+s_0}{K_M+s_0}
    +\dfrac{(K_M+s_0)^2(K_S+s_0)}{K_M^3}\right), \\
{\rm or} \nonumber \\
\varepsilon^{\S}_M &:= \varepsilon_{SSl}\left(\dfrac{K_S+s_0}{K_M+s_0}
    +\exp\left(\frac{s_0}{K_M}-1\right)\dfrac{(K_M+s_0)}{K_M}\right)
\end{align}
\textcolor{black}{with $K_S=k_{-1}/k_1$. }
Beyond these rigorously proven asymptotic estimates, numerical simulations suggest a sharper 
\textcolor{black}{bound
\begin{equation*}
    \varepsilon_{\rm opt}=\varepsilon_{SSl}\dfrac{K_S+s_0}{K_M+s_0}<\varepsilon_{SSl}.
\end{equation*}
Thus, while there may be problems with obtaining $K_S$ from experimental data, one may use 
$\dfrac{K_S+s_0}{K_M+s_0}<1$ and thus get error estimates involving only quantities that are 
controllable, or obtainable from \textcolor{black}{fitting progress curves or initial rate experiments}. In 
particular, this provides a mathematical foundation to the relevance of the Segel-Slemrod 
parameter $\varepsilon_{SSl}$.}

\subsection{Open challenges within the laboratory setting}
The Michaelis--Menten equation, 
\[
\dot s = - \dot p = -\dfrac{k_2 e_0 s}{K_M+s},
\]
involves two parameters, the Michaelis constant ($K_M$) and catalytic constant $(k_2)$ when 
the initial enzyme concentration ($e_0$) and initial substrate concentration ($s_0$) 
are known and can be controlled.

In principle, experimental scientists can estimate $k_2$ and $K_M$ via steady-state initial
rate experiments with the Michaelis--Menten equation, or steady-state progress curve
experiments with the Schnell--Mendoza equation \cite{SchMen}. However, there is a fundamental
problem with those parameter estimations. It requires to have {\em prior} knowledge of
the duration of transient $t_{\rm cross}$ and substrate depletion in the transient 
phase $s(t=t_{\rm cross})$, assuming sufficiently small $\varepsilon_{SSl}$. The fundamental 
goal of the present paper is to provide rigorous estimates for $t_{\rm cross}$, 
$s(t=t_{\rm cross})$ as well as $\varepsilon_{SSl}$ from a mathematical perspective. 

Generally, the role of our theoretical results is to provide consistency checks for 
experimental conclusions.  Our estimates for the crossing times involve only parameters 
that are controllable or amenable to determination by experiments, though challenges remains
in the unique estimation of $k_1$. In this respect, our mathematical results remain to be
explored in the experimental laboratory setting. By assuming sufficiently small 
$\varepsilon_{SSl}$, our theoretical results might make possible to obtain an educated 
guess for  $s_{\rm cross}$ by \eqref{scrossimple} in enzyme assays. By identifying the 
time when the guess for $s_{\rm cross}$ is attained, we could obtain an estimate for 
$t_{\rm cross}$, which in turn, with known $s_0$ and $K_M$, and equation \eqref{tcrosssimple}, 
could provide an estimate for $k_1$. Our results could also be used to check the 
consistency of experimental results by measuring the end of the transition time,
or the substrate depletion during the transient phase in steady-state experiments.

Interestingly, the same problem already was present with the Segel--Slemrod timescale 
$(k_1(s_0+K_M))^{-1}$, and it is actually an inherent feature of any parameter estimation 
that is solely based on the Michaelis--Menten equation \eqref{classmmeq}. The essential new aspect
of our work is that we obtained rigorous asymptotic expressions for the substrate loss in 
the transient phase, as well as for the approximation error, that only involve 
$\varepsilon_{SSl}=e_0/(K_M+s_0)$. But rigorous experimental protocols require further quantitative 
information --- e.g.\ about the onset of the slow time regime --- that is not readily available.
This remains an open problem for exploration in future work.

\section{Appendix}
\subsection{A quick-reference guide}
\textcolor{black}{
Tables~\ref{tab:PECs1} to \ref{tab:PECs3} provide essential constants and critical parameters 
for the Michaelis--Menten reaction mechanism. These are pivotal for designing accurate laboratory 
measurements, such as initial rates and progress curves, for reliable parameter estimation.}

\textcolor{black}{
Table~\ref{tab:PECs1} defines crucial steady-state constants of the Michaelis-Menten reaction.
These are the constants generally estimated in the laboratory.}

\textcolor{black}{
Table~\ref{tab:PECs2} introduces the foundational small parameter and fast timescale defined by 
Segel \& Slemrod. These concepts are instrumental in estimating the key parameters discussed in 
this paper.}

\textcolor{black}{
Table~\ref{tab:PECs3} serves as a quick reference for all critical parameters defined in this paper. 
Reliability indicators ($+$ for rigorous asymptotics, $++$ for asymptotics with rigorous upper and lower bounds) 
highlight the robustness of each estimate (details in Section 3). The Michaelis--Menten approximation
error bound, $\varepsilon_{SSl}$, is weaker than the previous one, $\varepsilon_{opt}$, but does 
not involve $K_S$, which may be unavailable. The accuracy of these estimates improves with smaller 
$\varepsilon_{SSl}$, a parameter initially unknown.}

\textcolor{black}{
In steady-state experiments, assuming a sufficiently small Segel-Slemrod parameter initially allows 
for estimating $K_M$ and $v_\infty$ using the Michaelis--Menten equation~\eqref{classmmeq}. If 
initial concentrations are known, this allows to calculate an estimate for $\varepsilon_{SSl}$
and performing a consistency check.}

\begin{table} 
\begin{tcolorbox}[arc=0pt,boxrule=0pt]
  \sisetup{group-minimum-digits = 4}
  \centering
  \caption{Michaelis--Menten reaction constants}
  \label{tab:PECs1}
  \begin{tabular}{ccccS[table-format=5]ll} 
    \toprule
    \thead{Parameter} & \thead{Expression} & \thead{Name} \\
    \midrule
    $K_M$ & $(k_{-1}+k_2)/k_1$ & Michaelis constant \\
    $K$ & $k_2/k_1$   & Van Slyke-Cullen constant \\
    $K_S$ & $k_{-1}/k_1=K_M-K$   & Dissociation constant  \\
    $v_\infty$ & $k_2e_0$  & Limiting rate  \\
    \bottomrule
  \end{tabular}
\end{tcolorbox} 
\end{table}

\begin{table} 
\begin{tcolorbox}[arc=0pt,boxrule=0pt]
  \sisetup{group-minimum-digits = 4}
  \centering
  \caption{Parameters from Segel \& Slemrod \cite{SSl}}
  \label{tab:PECs2}
  \begin{tabular}{ccccS[table-format=5]ll} 
    \toprule
    \thead{Parameter} & \thead{Expression} & \thead{Name} \\
    \midrule
    $\varepsilon_{SSl}$ & $e_0/(s_0+K_M)$ & Segel \& Slemrod small parameter\\
    $t_{SSl}$ & $[k_1(s_0+K_M)]^{-1}$  & Segel \& Slemrod fast timescale \\
    \bottomrule
  \end{tabular}
\end{tcolorbox} 
\end{table}

\begin{table} 
\begin{tcolorbox}[arc=0pt,boxrule=0pt]
  \sisetup{group-minimum-digits = 4}
  \centering
  \caption{Current parameters (lowest order terms only)}
  \label{tab:PECs3}
  \begin{tabular}{ccccS[table-format=5]ll} 
    \toprule
    \thead{Parameter} & \thead{Expression} & \thead{Description}&\thead{Reliability} \\
    \midrule
    $\Delta^{**}$ & $-\varepsilon_{SSl}\log \varepsilon_{SSl}$ & substrate depletion in transient & ++\\
    $t_{\rm cross}$ & $-t_{SSl}\log \varepsilon_{SSl}$ & QSS onset time& ++ \\
    $\varepsilon^{\ddagger}$ & $-\varepsilon_{SSl}\log \varepsilon_{SSl}$ & MM approximation error bound &++\\
    $\varepsilon_{opt}$ & $\varepsilon_{SSl}\cdot (s_0+K_S)/(s_0+K_M)$ & MM approximation error bound &+\\
    $\varepsilon_{SSl}$ & & MM approximation error bound &+\\
    \bottomrule
  \end{tabular}
\end{tcolorbox} 
\end{table}

\subsection{The Michaelis--Menten reaction mechanism with a low enzyme and substrate
binding rate constant ($k_1\to 0$)}
The case of low enzyme concentration in the Michaelis--Menten reaction mechanism is not 
the only one which leads to a singular perturbation reduction via Tikhonov and Fenichel. 
We can also obtain reductions in the limit $k_2\to 0$ (which will be discussed in future 
work) and in the limit $k_1\to 0$.\footnote{It is known from \cite{gwz} that these are all 
the possible ``small parameters'' for singular perturbation scenarios.} 

The case of low enzyme and substrate binding $k_1\to 0$ is of some interest since it 
represents the commonly expressed setting ``$s_0\ll K_M$'' in terms of singular perturbations 
(while letting $s_0\to 0$ does not). The arguments so far were motivated by the scenario with 
$e_0\to 0$, but all the estimates obtained in Sections~\ref{sec:review} and \ref{estsec} do hold, 
possibly upon rewriting some expressions involving $K_M$, without any restriction on the 
reaction rates and concentrations involved.

So here, we briefly summarize the pertinent results when $k_1\to 0$, while $e_0$ is 
bounded below.\footnote{Letting both $k_1$ and $e_0$ tend to zero leads to a degenerate 
Tikhonov-Fenichel reduction with trivial right hand side, which is of little interest.} 
Thus, $k_1=\varepsilon k_1^*$ with $\varepsilon\to 0$. The ``crossing Lemma'', 
{\bf Lemma~\ref{schahelem}} holds for all Michaelis--Menten type reaction mechanisms, so one 
may still employ the sQSS manifold given by $c=g(s)$ in \eqref{classsmeq} for the analysis 
of the system. Note that the first order approximation of the slow manifold is given by
\begin{equation*}
    c=\widehat g(s):=\dfrac{k_1e_0s}{k_{-1}+k_2},
\end{equation*}
but the discrepancy between $g$ and $\widehat g$ is of order $\varepsilon^2$, and the 
distinguished role of the sQSS manifold [defined by \eqref{classsmeq}] for the time 
course of complex concentration remains convenient in the analysis. One may also keep 
the Michaelis--Menten equation, in the version
\begin{equation*}
    \dot s =-\dfrac{k_1k_2e_0s}{k_{-1}+k_2+k_1s},
\end{equation*}
noting that the standard reduction procedure yields the right-hand side
\begin{equation*}
    -\dfrac{k_1k_2e_0s}{k_{-1}+k_2}= -\dfrac{k_1k_2e_0s}{k_{-1}+k_2+k_1s}+o(\varepsilon),
\end{equation*}
and for Tikhonov's theorem higher-order terms on the right-hand side are irrelevant.

It seems appropriate to take $\varepsilon_{RS}=\dfrac{k_1e_0}{k_{-1}+k_2}$ as a benchmark 
here. As noted, the relevant expressions obtained in Section~\ref{estsec} remain unchanged, 
but we record some asymptotics with the dots denoting higher order terms with respect to 
$\varepsilon_{RS}$:
\begin{equation*}
    \begin{array}{rcl}
         \varepsilon_{SSl}&=&\varepsilon_{RS}+\cdots\\
         t_{SSl}&=&\dfrac{1}{k_{-1}+k_2}+\cdots\\
         t_\ell^\dagger&=&\dfrac{1}{k_{-1}+k_2}\left(\log\dfrac{1}{\varepsilon_{RS}}+\log\dfrac{k_{-1}+k_2}{k_2}+\cdots\right)\\
         C^*&=&\dfrac{k_{-1}+k_2}{k_2}+\cdots\\
         t_u^\dagger(1)&=&\dfrac{1}{k_{-1}+k_2}\left(\log\dfrac{1}{\varepsilon_{RS}}+\log\dfrac{k_{-1}+k_2}{k_2}+\cdots\right)\\
         \varepsilon_\infty&=&\varepsilon_{RS}\cdot\dfrac{k_{-1}}{k_{-1}+k_2}+\cdots\\
    \end{array}
\end{equation*}
For the substrate depletion during the transient phase, one gets from {\bf Proposition~\ref{stuplem}} 
and {\bf Lemma~\ref{tuto1}}:
\begin{equation*}
    \dfrac{s_0-s_{\rm cross}}{s_0}\lesssim \varepsilon_{RS}\left(\log\dfrac{1}{\varepsilon_{RS}}+\log\dfrac{k_{-1}+k_2}{k_2}\right)+\cdots
\end{equation*}
We may summarize this by stating that in lowest order the dynamics is unaffected by initial 
substrate, in marked contrast to the low enzyme case.

\section*{Acknowledgments}
\textcolor{black}{We thank two anonymous reviewers for a thorough reading of the manuscript, and for constructive 
and helpful comments.}


\end{document}